\documentclass[12pt]{amsart}
\usepackage[top=1.2in, left=1.2in, right=1.2in, bottom=1.2in]{geometry}
\usepackage[utf8x]{inputenc}
\usepackage{amsmath,amsfonts,amssymb,epsfig}
\usepackage{delarray}
\usepackage{graphicx}
\newtheorem{thm}{Theorem}[section]
\newtheorem{theorem}[thm]{Theorem}
\newtheorem{cor}[thm]{Corollary}

\newtheorem{prop}[thm]{Proposition}
\newtheorem{proposition}[thm]{Proposition}

\newtheorem{defn}[thm]{Definition}
\newtheorem{definition}[thm]{Definition}
\theoremstyle{remark}
\newtheorem{remark}[thm]{Remark}
\numberwithin{equation}{section}
\newtheorem{example}[thm]{Example}

\newcommand{\cA}{\mathcal A}

\newcommand{\cC}{\mathcal C}
\newcommand{\cD}{\mathcal D}
\newcommand{\cF}{\mathcal F}
\newcommand{\cG}{\mathcal G}

\newcommand{\cK}{\mathcal K}
\newcommand{\cL}{\mathcal L}
\newcommand{\cM}{\mathcal M}
\newcommand{\cN}{\mathcal N}

\newcommand{\cU}{\mathcal U}
\newcommand{\cV}{\mathcal V}

\newcommand{\cY}{\mathcal Y}

\newcommand{\bbR}{\mathbb R}
\newcommand{\bbT}{\mathbb T}

\newcommand{\bbZ}{\mathbb Z}

\newcommand{\rank}{{\rm rank\ }}

\usepackage{color}

\begin{document}
\setcounter{tocdepth}{1}

\title{A conceptual approach to the problem of action-angle~variables}

\author{Nguyen Tien Zung}
\address{School of Mathematics, Shanghai Jiao Tong University (visiting professor), 800 Dongchuan Road, Minhang District, Shanghai, 200240 China}
\address{Institut de Math\'ematiques de Toulouse, UMR5219, 
Universit\'e Paul Sabatier, 118 route de Narbonne,
31062 Toulouse, France}
\email{tienzung@math.univ-toulouse.fr}

\begin{abstract}
In this paper we develop a general conceptual approach to the problem of 
existence of action-angle variables for dynamical systems, which establishes and
uses the \textit{fundamental conservation property of associated torus actions}:
anything which is preserved by the system is also preserved by the associated torus
actions. This approach allows us to obtain, among other things: a) the shortest 
and most easy to understand  conceptual proof of the classical 
Arnold--Liouville--Mineur theorem; b) basically all known results in the literature
about the existence of action-angle variables in various contexts can be recovered
in a unifying way, with simple proofs, using our approach; c) new results on 
action-angle variables in many different contexts, including systems on contact 
manifolds, systems on presymplectic and Dirac manifolds, 
action-angle variables near singularities,
stochastic systems, and so on. Even when there are no natural action variables, our approach still leads to useful normal forms for dynamical systems, which are not necessarily integrable.
\end{abstract}

\date{This version: February 2018, accepted for publication in ARMA}
\subjclass{37G05, 37J35,70H06,70H45}
\keywords{torus actions, action-angle variables, integrable system, 
contact, symplectic, Poisson structure, Dirac manifold, presymplectic}%

\maketitle

\tableofcontents

\section{Introduction}

Action-angle variables play a fundamental role in classical and quantum mechanics. They are
the starting point of the famous Kolmogorov--Arnold--Moser theory about the persistence of
quasi-preridicity of the motion of integrable Hamiltonian systems under perturbations
(see, e.g., \cite{BroerSevruyk-KAM2010,Dumas-KAM2014,KaPo-KAM2003}). 
They are also the starting point of geometric 
quantization rules which go back to the works of Bohr, Sommefeld, 
Epstein and Einstein in early 20th century
(see, e.g., \cite{BergiaNavarro-EinsteinQuantization2000}), and also of semi-classical quantization of integrable
Hamiltonian systems (see, e.g., \cite{San-Semiclassical2006}). 
The action functions are already written by contour integral formulas
(see Formula \eqref{eqn:MineurIntegral}) in the works of Burgers 
\cite{Burgers-Adiabatic1917},
Einstein \cite{Einstein-Quantization1917}, 
Levi-Civita \cite{LeviCivita-Adiabatic1927} and other physicists, and they also play the role of adiabatic  invariants in mechanics 
and physics. The quasi-periodicity of the movement of general proper 
integrable Hamiltonian systems in angle variables was 
discovered by Liouville in mid 19th century \cite{Liouville-1855}. 
The first essentially complete proof of the theorem about the 
existence of action-angle variables near a Liouville torus on a symplectic manifold, 
which is often called Arnold--Liouville  theorem, is due to the astrophysicist Mineur 
\cite{Mineur-AA1935,Mineur-AA1935-37}, who was also motivated by the quantization problem.
There are now many books which contain a proof of this theorem, see for example 
\cite{Arnold-Mechanics1989,GS-Symplectic1990,HoZe1994,LM-Symplectic1987, MM-1974}.
An interesting account on the early history of action-angle variables can be
found in a recent paper by Féjoz \cite{Fejoz-AA2013}.
 However, it seems to us that the available proofs in the literature 
are not yet ``very natural'': they contain arguments which are a bit tricky, 
and do not generalize easily to other contexts without a lot of additional work. 

There have been generalizations of Arnold--Liouville--Mineur 
action-angle variables theorem to various contexts, including noncommutatively 
integrable systems (see, e.g., \cite{Nekhoroshev-AA1972,MF-Noncommutative1978,DazordDelzant-AA1987}), systems on almost-symplectic
manifolds  \cite{FassoSansonetto-AA2007}, on contact manifolds 
(see, e.g., \cite{KhesinTaba-ContactIntegrable2010,Jovanovic_AAContact2011}), 
on Poisson manifolds 
\cite{LMV-AA2011}, and so on. 
Action-angle variables near singularities of integrable systems 
have also been studied (see, e.g., \cite{DufourMolino-AA1990, MirandaZung-NF2004,
San-Semiclassical2006,Zung-Integrable1996,Zung-Convergence2005,Zung-Torus2006}). 
However, there are many other natural contexts, including presymplectic and Dirac
manifolds, for which there was no systematic study of the existence of action-angle
variables, as far as I know.

In this paper, we develop a new general conceptual
approach for the study of existence of action-angle variables, 
which, in our opinion, is the most natural, easy to understand, 
and can be applied to a multitude of very different situations. If one is interested 
only in the classical Arnold--Liouville--Mineur theorem, then this approach will 
provide a very short and simple self-contained proof, only a few pages long. 
(Our proof in this paper finishes at page \pageref{ALM-Proof},
but that's because we have a long introduction and prove many other general 
results of independent interest along the way). Using our approach, one can also
recover practically all known results in the literature about the existence of
action-angle variable in various contexts, in a unifying way, with simple proofs,
and obtain a series of new results.

Our approach is based on the ``toric philosophy'' and consists of 3 parts:

i) \textit{Existence of associated torus actions for dynamical systems}. For example,
for integrable systems near Liouville tori, these actions are nothing but the Liouville
torus actions, which are provided by the classical Liouville theorem dating back 
to mid 19th century \cite{Liouville-1855}. For general analytic vector fields near singular points, 
these associated torus actions are exactly the ones appearing in the (formal or
analytic) Poincaré-Birkhoff normalization 
\cite{Zung-Convergence2002,Zung-Convergence2005}. 

ii) The \textit{fundamental conservation property} of these associated torus actions.
This is a fundamental property of dynamical systems, but whose proof is not very 
difficult and is provided in this paper. The philosophical 
idea behind this property comes from the notion of double commutant in algebra. 

iii) \textit{Simultaneous normalization} of the associated torus action together with
the underlying geometric structure preserved by it leads to action-angle variables
(if case of (quasi-)Hamiltonian systems) or other interesting normal forms.

The organization of this paper is as follows: 

In Section \ref{section:Conservation} we recall some basic notions about
integrable dynamical systems (which can be non-Hamiltonian), 
Liouville torus actions near Liouville tori, 
and establish the fundamental conservation property
for Liouville torus actions. (Theorems \ref{thm:TorusPreservesStructureI},
\ref{thm:TorusPreservesStructureIb}, \ref{thm:TorusPreservesStructureII}
and \ref{thm:TorusPreservesStructure4}).

In Section \ref{section:AA2Form} we show the existence of action-angle
variables for integrable Hamiltonian systems on manifolds with a differential
2-form, which is not necessarily closed nor nondegenerate. The classical
Arnold--Liouville--Mineur theorem, as well as results obtained by 
Nekhoroshev \cite{Nekhoroshev-AA1972} and Fassò--Sansonetto 
\cite{FassoSansonetto-AA2007} are presented as special cases of our general
results. We pay particular attention to (over-determined) action-angle variables for integrable Hamiltonian systems on \textit{presymplectic} manifolds, which can happen
quite often in practice (for example, by looking at the isoenergy submanifolds
of an integrable Hamiltonian system on a symplectic manifold), and which leads
to an interesting generalization of (integral) affine geometry, 
which we call (integral) co-affine geometry. Unlike affine manifolds, co-affine
manifolds have a lot of local invariants (e.g., curvature). A special case of
toric integrable Hamiltonian systems on presymplectic manifolds whose base spaces
are \textit{flat} co-affine manifolds was studied recently by Tudor
Ratiu and the author \cite{RatiuZung-Irrational2017}. 

In Section \ref{section:AAContact} we study the problem of action-angle variables
of contact manifolds, which is one of the problems posed by V.I. Arnold in 1995
\cite{Arnold-Problems2004}. Our result about action-angle variables
in the generic (transversal) case is a significant improvement of earlier results
by Banyaga--Molino \cite{BaMo-Contact1992} and Jovanovic \cite{Jovanovic_AAContact2011}.
We also obtain action-angle variables for the non-transversal case, 
which has not been studed by other authors, as far as we know.

Section \ref{section:AADirac} is the longest section of this paper, and is devoted to a study of integrable Hamiltonian systems on Dirac manifolds and their action-angle variables. There are two reasons why we are so interested in general Dirac manifolds:
i) There are many dynamical systems which cannot be written as Hamiltonian systems on symplectic or Poisson manifolds but which can be written as Hamiltonian systems on Dirac
manifolds; ii) The problem of action-angle variables on Dirac manifolds had not been
treated by any other author, as far as we know. In the special case, when the Dirac structure turns out to be a Poisson structure, we recover the main results of
Laurent--Miranda--Vanhaecke \cite{LMV-AA2011}, with a simpler proof. 

In Section \ref{Section:Singularities} we discuss the problem of action-angle variables
near singularities of dynamical systems and its relations with local normalizations
à la Poincaré--Birkhoff and associated torus actions 
\cite{Zung-Integrable1996,MirandaZung-NF2004,Zung-Convergence2002,
Zung-Toric2003,Zung-Convergence2005}. 

Finally, in Section \ref{Section:StochasticQuantum} we indicate how our general 
conceptual approach can lead to action-angle variables or at least interesting 
normal forms in other contexts: infinite-dimensional integrable systems, 
stochastic systems, quantum and semi-classical systems, etc.

\section{Fundamental conservation property of Liouville torus actions}
\label{section:Conservation}

\subsection{Integrable systems and Liouville tori}

Let us recall the following natural notion of integrability of dynamical systems which are not necessarily Hamiltonian
(see, e.g., \cite{Bogoyavlenskij-Extended1998,Zung-Convergence2002,Zung-Torus2006}): 

A $m$-tuple $(X_1\hdots,X_p, F_1,\hdots, F_q)$, 
where $p \geq 1, q \geq 0, p+q = m,$ $X_i$ are vector fields on
a $m$-dimensional manifold $M$ and $F_j$ are functions on $M$, is called an \textbf{\textit{integrable system}} of \textbf{\textit{type $(p,q)$}}
on $M$ if it satisfies the following commutativity and non-triviality conditions: \\
i) $[X_i,X_j] = 0 \ \ \forall i,j=1,\hdots,p$, \\
ii)  $X_i(F_j) = 0 \ \ \forall i \leq p, j \leq q$,  \\
iii) $X_1\wedge \hdots \wedge X_p \neq 0$ and $dF_1 \wedge \hdots \wedge dF_q \neq 0$ almost everywhere on $M$.

A dynamical system given by a vector field $X$ on a manifold $M$ is called 
\textbf{\textit{integrable}} if there is an integrable system  $(X_1\hdots,X_p, F_1,\hdots, F_q)$ of some type $(p,q)$ on $M$ with $X_1 = X$.

The above integrability notion is also called \textbf{\textit{non-Hamiltonian integrability}}. It does not mean that the system cannot admit any Hamiltonian structure, it simply means that we forget about the Hamiltonian structure, 
and look only at commuting flows and first integrals in the definition. 
If a Hamiltonian system on a symplectic manifold is integrable in the sense of Liouville or in noncommutative sense then it is also integrable in the above sense. But systems with non-holonomic constraints, which are a priori non-Hamiltonian, can also be integrable in the above sense.

By a \textbf{\textit{level set}} of an integrable system 
$(X_1\hdots,X_p, F_1,\hdots, F_q)$ we mean a connected component $N$ of
a joint level set
\begin{equation}
 \{F_1 = const,\hdots, F_q = const\}.
\end{equation}
Notice that, by definition, the vector fields $X_1,\hdots, X_p$ are tangent to the level sets of the system.
We will say that the system $(X_1\hdots,X_p, F_1,\hdots, F_q)$ is \textbf{\textit{regular}} at $N$ if $X_1\wedge \hdots \wedge X_p \neq 0$ 
and $dF_1 \wedge \hdots \wedge dF_q \neq 0$  everywhere on $N$. We will say
that the system is \textbf{\textit{proper}} if the map
$(F_1,\hdots,F_q): M \to \mathbb{R}^q$ is a proper topological map (so that
each level set is compact) and the system is regular at almost every level set.

The following theorem about the existence of a system-preserving torus action near a compact regular level set of an
integrable system is essentially due to Liouville \cite{Liouville-1855}: 

\begin{thm}[Liouville's theorem] \label{thm:Liouville}
Assume that $(X_1,\hdots,X_p,F_1,\hdots,F_q)$ in an integrable system of 
type $(p,q)$ on a manifold $M$  which
is regular at a compact level set $N$.  Then in a tubular neighborhood $\cU(N)$ there is, up to automorphisms of $\bbT^p$, 
a unique free torus action
\begin{equation}
 \rho: \bbT^p \times \cU(N) \to \cU(N)
\end{equation}
which preserves the system (i.e. the action preserves each $X_i$ and each $F_j$)
and whose orbits are regular level sets of the system.  In particular, $N$ is diffeomorphic to  $\bbT^p$, 
and
\begin{equation}
 \cU(N) \cong \bbT^p \times B^q
\end{equation}
with periodic coordinates $\theta_1 \pmod 1,\hdots,\theta_p \pmod 1$ on 
the torus $\bbT^p$ and coordinates
$(z_1,\hdots, z_q)$ on a $q$-dimensional ball $B^q$, such that $F_1,\hdots, F_q$ depend only on the variables
$z_1,\hdots, z_q,$ and the vector fields $X_i$ are of the type
\begin{equation}
 X_i = \sum_{j=1}^p a_{ij}(z_1,\hdots,z_q) \frac{\partial}{\partial \theta_j}.
\end{equation}
\end{thm}

A system of coordinates 
$$(\theta_1 \pmod 1,\hdots,\theta_p \pmod 1, z_1,\hdots, z_q)$$
on $\cU(N) \cong \mathbb{T}^p \times B^q$
given by the above theorem will be called a 
\textbf{\textit{Liouville system of coordinates}}.

The proof of the above theorem is absolutely similar to the case of integrable Hamiltonian systems on symplectic manifolds,
see, e.g., \cite{Bogoyavlenskij-Extended1998,Zung-Torus2006}. It consists of 
the following main points: 

1) The map
$(F_1,\hdots,F_q): \cU(N) \to \bbR^q$ from a tubular neighborhood of $N$ to $\bbR^q$ is a 
topologically trivial fibration by the level
sets, due to the compactness of $N$ and the regularity of $(F_1,\hdots,F_q)$ (attention: if $(F_1,\hdots,F_q)$ is not 
regular at $N$ then this fibration may be non-trivial and may be twisted even if the level sets are smooth); 

2) The vector fields
$X_1,\hdots,X_p$ generate a transitive action of $\bbR^p$ on the level sets near $N$, and the level sets are compact
and of dimension $p$, which imply that each level set $N_{f_1,\hdots, f_q}$ 
is a $p$-dimensional compact quotient of $\bbR^p$ by a discrete group $\Gamma$, i.e. a torus.

3) Consider a local section $S$ to the foliation by the level sets, i.e. a small disk which intersects each level set $N_{f_1,\hdots, f_q}$ near $N$
transversally at one point denoted by $s_{f_1,\hdots, f_q}$.
Let $\varphi^t_{X_i}$ be the flow of $X_i$. Composing the map 
$(t_1,\hdots, t_p) \mapsto \varphi^{t_1}_{X_1} \circ \hdots \circ 
\varphi^{t_p}_{X_p} (s_{f_1,\hdots, f_q})$ from $\mathbb{R}^q/\Gamma$ to
 $N_{f_1,\hdots, f_q}$  with an isomorphism from $\mathbb{R}^q/\mathbb{Z}^q$ to $\mathbb{R}^q/\Gamma$, one gets angular coordinates $\theta_1, \hdots, \theta_p$ which turn the $X_i$ into constant vector fields on
each level set $N_{f_1,\hdots, f_q}$.

Theorem \ref{thm:Liouville} shows that the  flow of the vector field $X=X_1$ of an integrable system
is quasi-periodic under some natural compactness and regularity conditions. This is the most fundamental geometrical
property of proper integrable dynamical systems.

Due to the above theorem, each $p$-dimensional compact level set $N$ of an integrable system of type $(p,q)$ on which the system is regular is called a \textbf{\textit{Liouville torus}}, and the torus $\mathbb{T}^p$-action
in a tubular neighborhood $\cU(N)$ of $N$ which preserves the system 
is called the  \textbf{\textit{Liouville torus action}}. Notice that 
this action is uniquely determined by the system, 
up to an automorphism of $\mathbb{T}^p$.

\subsection{Fundamental conservation property of Liouville torus actions}

The basic idea is the following meta-theorem, which is \textit{a new kind of conservation laws}: 

\begin{quote}
\textit{\textbf{Everything which is preserved by a dynamical 
system is also preserved by its associated torus actions}}. 
\end{quote}

In other words, the associated torus actions are \textit{double commutants} 
for dynamical systems. Some instances of this meta-theorem can be found, e.g., in \cite{Zung-Convergence2002,Zung-Convergence2005,Zung-Torus2006}.

In this subsection, we will turn the above meta-theorem into some rigorous 
theorems about the fundamental conservation property of Liouville torus actions, 
which play the role of associated torus actions for integrable systems
near Liouville tori.

Remark that the idea of double commutant torus actions having the fundamental conservation property is more general and also works for singularities of 
dynamical systems which are not necessarily integrable, and also for 
stochastic and quantum systems, see Section \ref{Section:Singularities} and 
Section \ref{Section:StochasticQuantum} of this paper.

 \begin{thm}[Fundamental conservation property, 1]  
 \label{thm:TorusPreservesStructureI}
Let $N$ be a Liouville torus of an integrable system  
$(X_1\hdots,X_p, F_1,\hdots, F_q)$ on a manifold $M$,
and $\cG \in \Gamma (\otimes^h TM \otimes^k T^*M)$
be a tensor field on $M$ which is preserved by all the
vector fields of the system: $\cL_{X_i} \cG = 0$ 
$\forall\, i=1,\hdots,p.$ Then the Liouville torus 
$\mathbb{T}^p$-action on a tubular neighborhood $\cU(N)$ of $N$
in $M$ also preserves $\cG$. 
\end{thm}

\begin{remark}
The above theorem first appeared in my unpublished preprint in 2012
(arXiv:1204.3865), and was also included in the lecture notes of a course
that I gave on integrable non-Hamiltonian systems in CRM Barcelona 
in 2013 (see \cite{BRZ-Integrable2016}), but it has not been published
in any peer-reviewed research journal before. The following
proof of the theorem borrows ideas from the theory of spectral sequences 
in algebraic topology.
\end{remark}

\begin{proof}
Fix a Liouville coordinate system  
$$(\theta_1 \pmod 1, \hdots, \theta_p \pmod 1, z_1,\hdots, z_q)$$ 
in a tubular neighborhood
$\cU(N)$ of $N$ as given by Theorem \ref{thm:Liouville}. The Liouville
torus action is generated by the vector fields 
$\dfrac{\partial}{\partial \theta_i}$.

We will make a filtration of the space 
$\Gamma (\otimes^h TM \otimes^k T^*M)$ of  tensor fields of covariant 
order $k$  and contravariant order $h$ as follows:

The subspace $T^{h,k}_s$ consists of sections of  
$\otimes^h TM \otimes^h T^*M$ whose expression in Liouville coordinates contains only  terms which are, \textit{up to a permutation of the factors},  of the type
\begin{equation}
\frac{\partial}{\partial \theta_{i_1}} \otimes \hdots \otimes \frac{\partial}{\partial \theta_{i_a}} \otimes
\frac{\partial}{\partial z_{j_1}} \otimes \hdots \otimes \frac{\partial}{\partial z_{j_b}} \otimes
d\theta_{i'_1} \otimes \hdots \otimes d\theta_{i'_c} \otimes
dz_{j'_1} \otimes \hdots \otimes dz_{j'_d}
 \end{equation}
with $b+c \leq s$. For example,
\begin{equation} 
 T^{h,k}_0 =  \left\{ \sum_{i,j'} f_{i,j'}\frac{\partial}{\partial \theta_{i_1}} \otimes \hdots \otimes \frac{\partial}{\partial \theta_{i_h}} \otimes
dz_{j'_1} \otimes \hdots \otimes dz_{j'_k} \right\} .
\end{equation}
Put $T^{h,k}_{-1}= \{0\}$.  It is clear that
\begin{equation}
 \{0\} = T^{h,k}_{-1} \subset T^{h,k}_ 0 \subset T^{h,k}_ 1 
\subset \hdots \subset T^{h,k}_{h+k} = \Gamma (\otimes^kTM \otimes^h T^*M).
\end{equation}

Observe that the vector fields $\dfrac{\partial}{\partial \theta_i}$ and
the differential forms $d z_j$ are invariant under the flow of any of the
vector fields $X_\alpha$ ($\alpha =1,\hdots, p$), 
while the Lie derivative $\cL_{X_\alpha} \dfrac{\partial}{\partial z_j}$ (respectively, $\cL_{X_i} d\theta_i$) is a combination of
vector fields $\dfrac{\partial}{\partial \theta_1}, \hdots, \dfrac{\partial}{\partial \theta_p}$ (respectively, of 1-forms $dz_1,\hdots, d z_q$).
It follows immediately from this observation 
that the above filtration is stable under the Lie derivative of the vector fields $X_1,\hdots, X_p$, i.e., we have
\begin{equation}
\cL_{X_{\alpha}}\Lambda \in T^{h,k}_s \ \ \forall  s=0,\hdots,k+h, \ \forall \Lambda \in T^{h,k}_s,\ \forall \alpha = 1,\hdots,p.
\end{equation}

Since $\cL_{X_\alpha} \cG = 0$ for all $\alpha=1,\hdots,p$ by our hypothesis, and the Liouville torus action
commutes with the vector fields $X_\alpha$,  we also have that $\cL_{X_\alpha} \overline \cG = 0,$ 
where the overline means the average of a tensor with respect to the Liouville torus action.
Thus we also have
\begin{equation}
\cL_{X_\alpha}\hat {\cG} = 0  \ \forall \alpha =1,\hdots, p,
\end{equation}
where 
\begin{equation}
 \hat \cG = \cG - \overline \cG
\end{equation}
has average equal to 0.

We will show by induction that $ \hat \cG \in T^{h,k}_{s}$
for every $s$ going down from $h+k$ to $-1$. Of course we have
$ \hat \cG \in T^{h,k}_{h+k} = \Gamma (\otimes^kTM \otimes^h T^*M)$,
and at the end of the induction process we will get
$ \hat \cG \in T^{h,k}_{-1} = \{0\}$, i.e. $\hat \cG = 0$. 

Assume that we already have $ \hat \cG \in T^{h,k}_{s}$ for some $s$ with
$h+k  \geq s \geq 0$. Let us show that $ \hat \cG \in T^{h,k}_{s-1}$.

Consider a possible monomial 
term $\Psi$ in the expression of $\hat \cG$ which belongs to
$T^{h,k}_{s}$ but does not belong to $T^{h,k}_{s-1}$.
Up to a permutation of the factors, $\Psi$ is of the type 
$$\Psi = \psi \frac{\partial}{\partial \theta_{i_1}} \otimes \hdots \otimes \frac{\partial}{\partial \theta_{i_a}} \otimes
\frac{\partial}{\partial z_{j_1}} \otimes \hdots \otimes \frac{\partial}{\partial z_{j_b}} \otimes
d\theta_{i'_1} \otimes \hdots \otimes d\theta_{i'_c} \otimes
dz_{j'_1} \otimes \hdots \otimes dz_{j'_d}$$
such that $b+c =s$, with some coefficient function $\psi$.

According to the observation that we made above, 
in the decomposition of the Lie derivative $\cL_{X_\alpha} \Psi$
by the Leibniz rule, all the terms belong to $T^{h,k}_{s-1}$ 
except maybe the term
$$X_\alpha (\psi) \frac{\partial}{\partial \theta_{i_1}} \otimes \hdots \otimes \frac{\partial}{\partial \theta_{i_a}} \otimes
\frac{\partial}{\partial z_{j_1}} \otimes \hdots \otimes \frac{\partial}{\partial z_{j_b}} \otimes
d\theta_{i'_1} \otimes \hdots \otimes d\theta_{i'_c} \otimes
dz_{j'_1} \otimes \hdots \otimes dz_{j'_d}$$
which is a monomial term belonging to $T^{h,k}_{s} \setminus T^{h,k}_{s-1}$ 
if $X_\alpha (\psi) \neq 0$. But $\cL_{X_\alpha}\hat {\cG} = 0$, which 
implies in particular that $\cL_{X_\alpha}\hat {\cG} \in T^{h,k}_{s-1}$ and 
it cannot contain any monomial term in $T^{h,k}_{s} \setminus T^{h,k}_{s-1}$,
and so we must have $X_\alpha (\psi) = 0$ (for every $\alpha=1,\hdots, p$), i.e.
$\psi$ is invariant with respect to the vector field $X_1,\hdots, X_p$.
It means that $\psi$ is constant on each Liouville torus. On the other hand,
by the definition of $\hat G$, the mean value of $\psi$ on each
Liouville torus is zero, so in fact $\psi$ is identically zero, and there
is no monomial term of $\cG$ in $T^{h,k}_{s} \setminus T^{h,k}_{s-1}$, i.e.
we have $\cG \in T^{h,k}_{s-1}$.

Thus, we have shown by induction that $\hat \cG = 0,$ i.e. $\cG = \overline \cG$  
is invariant with respect to the Liouville torus action.
\end{proof}

Theorem \ref{thm:TorusPreservesStructureI} admits the following slightly stronger version:

 \begin{thm}[Fundamental conservation property, 2]
 \label{thm:TorusPreservesStructureIb}
Let $N$ be a Liouville torus of an integrable system  
$(X_1\hdots,X_p, F_1,\hdots, F_q)$ on a manifold $M$,
such that in a tubular neighborhood $\cU(N) \cong \mathbb{T}^p \times B^q$
of $N$ the set of points ${\bf z} \in B^q$ such that the orbits
of the vector field $X_1$ on the Liouville torus 
$\mathbb{T}^p \times \{ {\bf z} \}$ are dense in the torus (i.e.,
the flow of $X_1$ is totally irrational on the torus) is a dense
subset of $B^q$. Then any tensor field 
$\cG \in \Gamma (\otimes^h TM \otimes^k T^*M)$ which is invariant
with respect to $X_1$ is also invariant with respect to the 
 Liouville torus $\mathbb{T}^p$-action on $\cU(N)$.
\end{thm}

Notice that in Theorem \ref{thm:TorusPreservesStructureIb}, we do not require
$\cG$ to be invariant with respect to $X_2,\hdots, X_p$. We only require it to be 
invariant with respect to the original vector field $X = X_1$ of an integrable
dynamical system, and this invariance is often automatically obtained in practice,
e.g., a Hamiltonian vector field will automatically preserve the symplectic structure, but the additional vector fields are not required to be symplectic. 

The proof of Theorem \ref{thm:TorusPreservesStructureIb} is absolutely
similar to the proof of Theorem \ref{thm:TorusPreservesStructureI}
and is based on the observation that, under the above assumptions about
$X_1$, any function which is invariant with respect to $X_1$ is also
invariant with respect to the Liouville torus action in $\cU(N)$.

We will call the condition about dense orbits in Liouville tori
imposed on the vector field $X_1$ in Theorem \ref{thm:TorusPreservesStructureIb}
the \textbf{\textit{complete irrationality}} condition.
Theorem \ref{thm:TorusPreservesStructureIb} is stronger than Theorem
\ref{thm:TorusPreservesStructureI}, because Theorem \ref{thm:TorusPreservesStructureI}
can be deduced from it by changing $X_1$ to a new vector field $X_1' = \sum c_i X_i$,
which is a linear combination of $X_1,\hdots, X_p$ with constant coefficients, and which
satisfies the complete irrationality condition. For integrable Hamiltonian systems
in the sense of Liouville, this complete irrationality is implied by the Kolmogorov's
nondegeneracy condition in K.A.M theory, see, e.g., 
\cite{Dumas-KAM2014,Zung-Kolmogorov1996}.

The above theorem can be applied to many kinds of underlying geometric structures which are preserved by the systems, e.g. volume form (isochore systems), Riemannian
metric, Nambu structure, symplectic or Poisson structure (Hamiltonian systems), 
symmetry groups or algebras (generated by vector fields, 
which are considered as tensors), and so on.
However, there are some geometric structures, e.g., contact distributions and
Dirac structures, which can't be written as tensors. To deal with them, we have
to extend Theorem \ref{thm:TorusPreservesStructureI} to the case of \textit{subbundles of natural vector bundles} 
preserved by the system, as will be explained below.

We will say that a tensor field $\cG$ on $M$ 
is \textbf{\textit{conformally conserved}} 
by the system $(X_1\hdots,X_p, F_1,\hdots, F_q)$
(or \textbf{\textit{conformally invariant}}) 
if for each $i=1,\hdots,p$ there is a (smooth) 
function $f_i$ on $M$ such that $\cL_{X_i} \cG = f_i. \cG$.

\begin{theorem}[Fundamental conservation property, 3]
\label{thm:TorusPreservesStructureII}
With the above notations, if a tensor field $\cG$ is conformally conserved by 
the integrable system $(X_1\hdots,X_p, F_1,\hdots, F_q)$ in a neighborhood of 
a Liouville torus, or if $X_1$ satisfies the complete irrationality condition and
$\cG$ is conformally invariant with respect to $X_1$,  
then $\cG$ is also conformally invariant with respect to 
the Liouville $\bbT^p$-action.
\end{theorem}

\begin{proof}
The proof of Theorem \ref{thm:TorusPreservesStructureII} can be reduced
to the proof of Theorem \ref{thm:TorusPreservesStructureI}, 
by multiplying  $\cG$ by an appropriate function.

Indeed, consider the case when $X_1$ satisfies the complete irrationality condition,
and let $\cG \in \Gamma (\otimes^h TM \otimes^k T^*M)$  be conformally invariant
with respect to $X_1$: $\cL_{X_1} \cG = f \cG$ for some function $f$. 
(The case when $\cG$ is invariant with respect to $X_1,\hdots, X_p$ but wihout
the irrationality condition can be reduced to this case by taking a linear
combination $X_1' = \sum_{i=1}^p c_i X_i$ with appropriate constant coefficients
$c_i$ so that $X_1'$ is completely irrational and preserves $\cG$).

We assume that $\cG \neq 0$, and  want to show that there is a function $g$
such that $\exp(g) \cG$ is invariant with respect to $X_1$. 
By conformal invariance, we have 
$\cL_{X_1} (\exp(g)\cG) = X_1(g) \exp(g) \cG + f \exp(g) \cG$,
so the equation to solve is 
$$ X_1(g) + f =0, $$
which of course locally has a solution, which is unique up to a function
which is constant on the orbits of $X_1$. The problem is that
maybe it doesn't have a global smooth solution in $\cU(N)$. So we have to show that
a global solution in fact exists.

With respect to the filtration given in
the proof of Theorem \ref{thm:TorusPreservesStructureI}, there is a number
$s$ such that $\cG \in  T^{h,k}_s \setminus T^{h,k}_{s-1}$. Let
\begin{equation}
h \frac{\partial}{\partial \theta_{i_1}} \otimes \hdots \otimes 
\frac{\partial}{\partial \theta_{i_a}} \otimes
\frac{\partial}{\partial z_{j_1}} \otimes \hdots \otimes 
\frac{\partial}{\partial z_{i_b}} \otimes
d\theta_{i'_1} \otimes \hdots \otimes d\theta_{i'_c} \otimes
dz_{j'_1} \otimes \hdots \otimes dz_{j'_d}
 \end{equation}
with $b+c=s$, up to a permutation of the factors, 
be a term of $\cG$ of highest filtration degree $b+c =s$
with non-zero coefficient function $h$. Then the term of the 
same type in $\cL_{X_1}\cG$ has coefficient $X_1(h)$. 
Since $\cL_{X_1} \cG = f\cG$ and $h \neq 0$ we must have $f = X_1(h)/h$. 
It implies (also due to the fact that $X_1$ is completely irrational) 
in particular that $h \neq 0$ everywhere in a small tubular neighborhood
$\cU(N)$ of $N$, and that the solutions in $\cU(N)$ of the equation 
$ X_1(g) + f =0$ are $g = - \ln(|h|) + const.$, and we are done.
\end{proof}

Given a manifold $M$, the vector bundles on $M$ which can be obtained from the tangent and cotangent bundles $TM$, $T^*M$
and the trivial bundle $\mathbb{R} \times M$ by operations of
taking sums and tensor products will be called
\textbf{\textit{natural vector bundles}} over $M$.

\begin{theorem}[Fundamental conservation property, 4]
\label{thm:TorusPreservesStructure4}
Let $\mathcal{V}$ be a subbundle of a natural vector bundle over a
manifold $M$, which is conserved by an integrable system on $M$, 
or by the completely irrational
vector field $X_1$ of the system. Then the  
Liouville torus action near any Liouville torus of the system preserves $\cV$.

\end{theorem}

\begin{proof}
Theorem \ref{thm:TorusPreservesStructure4} can be reduced to Theorem
\ref{thm:TorusPreservesStructureII} by using a tensor field which is a
volume element on each fiber of the subbundle $\mathcal{V}$. (The exterior product of the
components of a basis of a vector space is a contravariant volume element
of that vector space). This tensor field is not invariant with respect 
to the system in general, but it will be conformally invariant, and it characterizes the subbundle $\mathcal{V}$.

There is a small technical problem: maybe this volume tensor field 
cannot be defined globally due to the possible non-orientability of the 
vector subbundle $\mathcal{V}$ (i.e., the holonomy does not preserve the orientation of the fibers). But this situation can be remedied easily by taking a double covering of the system.
\end{proof}

\begin{remark}
Any tensor field can be viewed as a linear map from a natural
vector bundle to another natural vector bundle (e.g., a differential 
$k$-form on $M$ can be viewed as an anti-symmetric linear map from the bundle 
$\otimes^k TM$ to the trivial bundle $\mathbb{R}$ on $M$, or also a linear map
from $\otimes^{k-1} TM$ to $T^*M$, etc.), and the graph of such a linear map 
(which also characterizes our tensor field) is a subbundle of the sum
of the two bundles, and this sum is also a natural bundle. So any tensor field
may be viewed as a subbundle of a natural vector bundle. 
For example, both differential 2-forms and
2-vector fields can be viewed as 2-dimensional subbundles of $TM \oplus T^*M$.
But the converse is not true.
\end{remark}

Theorem \ref{thm:TorusPreservesStructure4} can be applied to 
invariant distributions, e.g. nonholonomic constraints and contact structures,
to systems on Dirac manifolds (which is one of the origins of this paper),
and so on.

Another generalization of the fundamental conservation property is to
differential operators which are preserved by the system. In fact, we have
the following result, which was proved recently by N.T. Thien and the author in 
\cite{ZungThien-Reduction2015} using the same method of filtration, 
and which is useful for the study of reduction
and  action-angle variables for stochastic and quantum systems:

\begin{theorem}[\cite{ZungThien-Reduction2015}] 
If $\Lambda$ is a linear differential operator on $M$ 
which is preserved by an integrable system $(X_1,\hdots, X_p,F_1,\hdots, F_q)$, 
or by a completely irrational component $X_1$ of the system, then 
$\Lambda$ is also invariant with respect to the Liouville torus action
in the neighborhood of  any Liouville torus of the system.
\end{theorem}

\section{Action-angle variables on manifolds with a differential 2-form}
\label{section:AA2Form}

In this section, we will give a very simple proof 
of the Arnold--Liouville--Mineur theorem about the existence of 
action-angle variables near an invariant torus of
a Hamiltonian system which is integrable in the classical sense 
of Liouville. In fact, we will obtain action-angle variables 
in a more general setting, for systems on manifolds with a differential 
2-form which is not necessarily closed or nondenegerate, and the 
Arnold--Liouville--Mineur theorem is just a particular case of this
more general result.

Consider a manifold $M$ together with an arbitrary given 
differential 2-form $\omega$, which is not necessarily closed nor
nondegenerate. We will say that a vector field $X$ on $(M,\omega)$ 
is a \textbf{\textit{Hamiltonian vector field}} of a function $H$
with respect to $\omega$ if 
it satisfies the following two equations:
\begin{equation} \label{eqn:Hamilton1}
X \lrcorner \omega = - dH
\end{equation}
and
\begin{equation}\label{eqn:Hamilton2}
\cL_{X} \omega = 0. 
\end{equation}

When $d\omega = 0$ then Equality \eqref{eqn:Hamilton2} can be omitted 
because it is a consequence of Equality \eqref{eqn:Hamilton1}. When
$d\omega \neq 0$ then, under the assumption that  Equality \eqref{eqn:Hamilton1} holds,  Equality \eqref{eqn:Hamilton2} 
is equivalent to the fact that $X$ lies in the kernel of $d\omega$:
\begin{equation}\label{eqn:Hamilton3}
X \lrcorner d \omega = 0. 
\end{equation}

\begin{definition} \label{defn:IntegrableHam2Form}
We will say that an integrable system $(X_1,\hdots, X_p, F_1,\hdots, F_q)$
on a manifold $(M,\omega)$, where $\omega$ is a given differential 2-form on $M$, is an \textbf{\textit{integrable Hamiltonian system}} of type $(p,q)$ if there are $p$ functions $H_1, \hdots, H_p$ on $(M,\omega)$ such that $X_i$ is a Hamiltonian vector field of $H_i$ with respect to $\omega$
for all $i=1,\hdots, p$.
\end{definition}

It follows immediately from Definition \ref{defn:IntegrableHam2Form} that if we have a proper integrable system $(X_1,\hdots, X_p, F_1,\hdots, F_q)$ 
which is Hamiltonian on $(M,\omega)$ with the aid
of Hamiltonian functions $H_1,\hdots, H_p$, then these functions  
$H_1,\hdots, H_p$ are common first integrals of the system, i.e.
$X_i(H_j) = 0$ for all $i,j =1,\hdots,p$. Indeed, $X_i(H_j) = \omega(X_i,X_j)$ 
is invariant with respect to the vector fields of 
the system, hence this function is constant on each Liouville torus.
Moreover, the mean value of $X_i(H_j)$ on each Liouville torus (with respect to the Liouville $\mathbb{T}^p$-action) is zero, because
$$ X_i(H_j) = \sum_k a_{jk}(z)\dfrac{\partial H_j}{\partial \theta_k}$$
in a Liouville coordinate system,
which implies that the integration of $X_i(H_j)$ over a Liouville torus with
respect to the invariant volume form $d\theta_1 \wedge \hdots \wedge d\theta_p$
is zero. Hence $X_i(H_j)$ is identically zero.

The equality $X_i(H_j) = \omega(X_i,X_j) = 0$ also means that, for proper integrable 
Hamiltonian systems on $(M,\omega)$ (where the 2-form $\omega$ may be non-closed and 
degenerate), the Liouville tori are isotropic with respect to $\omega$. As an immediate consequence, we get the following inequalities:
\begin{equation} \label{eqn:pq}
p \leq \dim M - \dfrac{\rank \omega}{2}, \quad q \geq \dfrac{\rank \omega}{2}.
\end{equation}
In particular, if $\omega$ is nondegenerate then $\rank \omega = \dim M$ and
$p \leq \dfrac{\dim M}{2}.$

\begin{proposition} 
\label{prop:LiouvilleHamiltonianAction0}
If $(X_1,\hdots, X_p, F_1,\hdots, F_q)$ is a proper
integrable Hamiltonian system on $(M,\omega)$, where $\omega$ is an
arbitrary differential 2-form, then the Liouville $\bbT^p$-action 
in the neighborhood of any Liouville torus is a Hamiltonian action
with respect to $\omega$, i.e. the components of this action are 
generated by Hamiltonian vector fields.
\end{proposition}

\begin{proof}
Let $Z_1,\hdots, Z_p$ be the generators of the Liouville $\bbT^p$-action
near some Liouville torus $N$. 
Since $X_i \lrcorner d\omega = 0$  $\forall\ i$, we also have 
$Z_i \lrcorner d\omega = 0$ $\forall\ i$. Together with 
$\cL_{Z_i} \omega = 0$ (according to the fundamental conservation
property), it implies that $d (Z_i \lrcorner \omega) = 0,$ i.e.,
$Z_i \lrcorner \omega$ is closed in a neighborhood $\cU(N)$ of any Liouville
torus $N$, hence it is exact (because its pull-back to a Liouville torus is zero), therefore there exists a $\bbT^p$-invariant function $\mu_i$ 
such that $Z_i \lrcorner \omega = - d\mu_i$. Thus, 
the Liouville $\bbT^p$-action is Hamiltonian with momentum map $(\mu_1, \hdots, \mu_p).$ 
\end{proof}

Let $(\tilde{\theta_1},\hdots \tilde{\theta_p},z_1,\hdots,z_q)$ 
be a Liouville coordinate system on $\cU(N)$ such that 
$\dfrac{\partial}{\partial \tilde{\theta_i}} = Z_i$. The 2-form
$$\beta = \omega - \sum_{i=1}^p d\mu_i \wedge d\tilde{\theta_i}$$
is $\bbT^p$-invariant, and $Z_i \lrcorner \beta = 0$ $\forall i=1,\hdots, p$, hence
$\beta$ is a basic 2-form with respect to the $\bbT^p$-action, i.e., we can write $\beta$ as  (the pull-back of) a 2-form on the space of Liouville tori:
$$\beta = b_{ij}(z) dz_i \wedge dz_j .$$ 
(There is no $\tilde{\theta_i}$ in the expression).
Thus we get the following normal form, similar to the one obtained by Fass\`o and
Sansonetto  in 2007 \cite{FassoSansonetto-AA2007}:
\begin{equation} \label{eqn:NF2Form1} 
\omega = \sum_{i=1}^p d\mu_i(z) \wedge d\tilde{\theta_i} + \sum_{1\leq i<j \leq q} b_{ij}(z) dz_i \wedge dz_j. 
\end{equation}

In the above normal form, the coordinates $(\mu_1,\tilde{\theta_1},\hdots, \mu_p, \tilde{\theta_p})$
may be viewed as a kind of \textit{partial action-angle variables}. Remark that, 
when $\omega$ is degenerate, the functions $\mu_1,\hdots, \mu_p$ are not functionally 
independent in general (see the subsection about the presymplectic situation below). 

When $\omega$ is nondegenerate (it is called an \textit{almost-symplectic} form in 
this case) then the linear independence of $Z_1,\hdots, Z_p$ implies the functional independence
of $\mu_1,\hdots,\mu_p$, and so we may choose $z_1 = \mu_1,\hdots, z_p = \mu_p$, and
the normal form \eqref{eqn:NF2Form1} becomes

\begin{equation} \label{eqn:NF2Form2} 
\omega = \sum_{i=1}^p d z_i \wedge d\tilde{\theta_i} + \sum_{1\leq i<j \leq q} b_{ij}(z) dz_i \wedge dz_j. 
\end{equation}

In this nondegenerate (but not necessarily closed) case, we may call
$(z_1,\tilde{\theta_1},\hdots, z_p, \tilde{\theta_p})$ (generalized partial) 
action-angle variables, and  the coordinates $(z_{p+1},\hdots, z_q)$
(if there are any, i.e. if $q > p$) are additional variables. The 2-form
$\beta = \sum_{1\leq i<j \leq q} b_{ij}(z) dz_i \wedge dz_j$ is a kind of \textit{\textbf{magnetic term}} 
(which is not necessarily closed). 

Under some additional conditions on the system or on the 2-form $\omega$, starting from
the normal forms \eqref{eqn:NF2Form1} and \eqref{eqn:NF2Form2} 
we will get more refined normal forms. 

\subsection{Action-angle variables for Liouville-integrable systems on (almost) symplectic manifolds}

When $p = q = \dim M / 2$ and the 2-form $\omega$ is nondegenerate, the normal form 
\eqref{eqn:NF2Form2} reads
\begin{equation} \label{eqn:NF2Form3} 
\omega = \sum_{i=1}^n d z_i \wedge d\tilde{\theta_i} + 
\sum_{1\leq i<j \leq n} b_{ij}(z) dz_i \wedge dz_j,
\end{equation}
in a coordinate system of \textit{generalized action-angle variables} 
$(z_1,\tilde{\theta_1},\hdots, z_n, \tilde{\theta_n})$ on $M^{2n}$. (We say ``generalized''
because there is still the magnetic term 
$\beta =\sum_{1\leq i<j \leq n} b_{ij}(z) dz_i \wedge dz_j$). The Liouville torus action
is generated by $(\dfrac{\partial}{\partial \tilde{\theta_1}},\hdots, 
\dfrac{\partial}{\partial \tilde{\theta_n}})$ in this
coordinate system.

If, moreover, $\omega$ is closed (i.e. it is really a symplectic form), then the
magnetic term $\beta = \sum_{1\leq i<j \leq q} b_{ij}(z) dz_i \wedge dz_j$ in \eqref{eqn:NF2Form3}
is also closed, $d\beta = 0$, hence locally exact by Poincaré's lemma, i.e., we can write
$\beta = d (\sum a_i(z) dz_i)$. It follows that
$\omega = \sum dz_i \wedge d\tilde{\theta_i} - \sum dz_i \wedge d a_i(z)$,
i.e., we have:
\begin{equation} \label{eqn:NF2Form4}
\omega = \sum_{i=1}^n dz_i \wedge d \theta_i,
\end{equation}
where $\theta_i = \tilde{\theta_i} - a_i$. 

It means that $(z_1,\theta_1,\hdots,z_n,\theta_n)$ is a system of action-angle variables 
(in a neighborhood $\cU(N)$ of an arbitrary given Liouvile torus $N$), 
and the first integrals $F_1,\hdots, F_n$ of the system in $\cU(N)$ depend only on the
variables $(z_1,\hdots,z_n)$ because $X_i(F_j) = 0$ $\forall\ i, j$. 
We have completed the proof of the classical 
Arnold--Liouville--Mineur theorem. \label{ALM-Proof}

\begin{remark} 
We may compare the present proof of existence of action-angle variables 
in the symplectic case with the one given in \cite{AKN-1987} and other books. 
There, given a Liouville system of coordinates, an analysis through Poisson brackets of  the way to make standard
the symplectic form $\omega$ 
leads to looking for an appropriate coordinate change of the form
$(z, \tilde{\theta}) \mapsto (\mu(z), \tilde{\theta}-a(z))$ (our notations). 
The action functions $\mu_i$ are usually defined via the integral formula 
\eqref{eqn:MineurIntegral}, and then one shows that the Hamiltonian vector fields
of these action functions are equal to $\dfrac{\partial}{\partial \theta_i}$ 
by using intricate computations (which work but it's still somewhat of a mistery why they work) and/or clever geometric arguments (e.g., identifying the universal covering of a tubular neighborhood of the Liouville torus $N$ with
the cotangent bundle of a local section to the torus fibration, and 
identifying the preimages of the section in this covering with the 
graphs of closed 1-forms over the section in the cotangent bundle picture, 
using general results from symplectic geometry).  
In the present proof, everything is natural: both the actions and
the angles become a natural byproduct of the Hamiltonianity of the Liouville 
torus action: $\mu$ is its momentum map
while the rotation $\tilde{\theta} \to \theta = \tilde{\theta} − a(z)$, 
which makes the section $\{\theta = const.\}$ of the Liouville torus fibration Lagrangian (when $\omega$ is closed) 
is the contribution of the magnetic term  which appears in the normal form 
for the most general 2-form $\omega_i$. 
\end{remark}

\subsection{Super-integrable systems on symplectic manifolds} 

Consider now the case when $\omega$ is a symplectic form and the 
integrable system $(X_1,\hdots, X_p, F_1,\hdots, F_q)$ 
is Hamiltonian on $(M,\omega)$, but with $p < n = \dfrac{1}{2} \dim M$
and $q = 2n - p > n$.
(According to \eqref{eqn:pq}, we cannot have $p > n$). 
Such systems are often called \textit{\textbf{super-integrable}} 
in the literature, 
because there are more first integrals than in the Liouville-integrable case. 

A particularly important class of super-integrable
systems are the so-called \textit{non-commutatively integrable systems} 
introduced by
Fomenko and Mischenko \cite{MF-Noncommutative1978} in 1978, which also appeared
in a book by Abraham and Marsden \cite[Exercise 5.2I]{AbMa1978} 
at around the same time. 

Action-angle variables for super-integrable Hamiltonian systems were studied by Nekhoroshev
\cite{Nekhoroshev-AA1972}. His main result says that, in this case, the symplectic
form $\omega$ has the following normal form:
\begin{equation} \label{eqn:NF2Form5}
\omega = \sum_{i=1}^n d\mu_i \wedge d \theta_i + \sum_{i=1}^{n-p} dx_i \wedge dy_i
\end{equation}
in a coordinate system $(\theta_1 \pmod 1, \mu_1,\hdots, \theta_p \pmod 1, \mu_p, x_1, y_1,\hdots,
x_{n-p}, y_{n-p})$. 

Formula \eqref{eqn:NF2Form5} is of course a particular case of Formula \ref{eqn:NF2Form2}.
In order to obtain Formula \eqref{eqn:NF2Form5}, one chooses a coisotropic section $S$ 
to the Liouville torus fibration in $\cU(N)$ and choose the coordinates $\theta_i$ such that
$Z  = \dfrac{\partial}{\partial \theta_i}$ and $\theta_i = 0$ on $S$. Then
$\beta = \omega - \sum_{i=1}^n d\mu_i \wedge d \theta_i$ is $\mathbb{T}^p$-invariant and its
pull-back to $S$ is a closed basic 2-form with respect to the isotropic foliation on $S$.
(This isotropic foliation is generated by the Hamiltonian vector fields 
of $\theta_1,\hdots, \theta_p$). On the quotient space of $S$ by the isotropic foliation, $\beta$
becomes nondegenerate, and hence can be written as $\beta = \sum_{i=1}^{n-p} dx_i \wedge dy_i$
by Darboux's theorem. These coordinates $(x_i, y_i)$ can be pulled back to $S$ and then extended
to $\cU(N)$ in a $\mathbb{T}^p$-invariant way, and we get a coordinates system
$(\theta_1 \pmod 1, \mu_1,\hdots, \theta_1 \pmod p, \mu_p, x_1, y_1,\hdots,
x_{n-p}, y_{n-p})$ in which $\omega$ is given by Formula \eqref{eqn:NF2Form5}.

\subsection{Presymplectic action-angle variables} 

Consider now the case when $\omega$ is a presymplectic structure
(i.e., $d\omega = 0$) of constant rank, and $p$ is maximal possible,
i.e. 
\begin{equation}
p = \dim M - \dfrac{1}{2} \rank \omega \; \text{and} \; 
q = \dfrac{1}{2} \rank \omega. 
\end{equation}

Such a situation can happen quite often in practice. For example, starting
from a Liouville-integrable Hamiltonian system on a symplectic manifold
$M^{2n}$, one fixes some regular values of some first integrals $F_1,\hdots, F_d$. Then one gets a presymplectic manifold 
\begin{equation}
\hat{M} = \{x\in M \mid F_1(x) = c_1,\hdots, F_d(x) = c_d\},
\end{equation}
whose presymplectic form $\omega$ (which is the pull-back of the
symplectic form from $M$) has rank $2n-d$,
and the restricted integrable Hamiltonian 
system on it with $p =n, q = n-d$.

By the same arguments as in the Liouville-integrable case on symplectic
manifolds, one sees that in this regular presymplectic case, the presymplectic form still has the form
\begin{equation}
\omega = \sum_{i=1}^p d\mu_i \wedge d \theta_i
\end{equation}
in a \textbf{\textit{over-determined}} \textit{action-angle coordinate system}  $$(\theta_1 \pmod 1, \mu_1,\hdots, \theta_p \pmod 1, \mu_p)$$
on $\cU(N) \cong \mathbb{T}^p \times B^q$:
the functions $\mu_1,\hdots,\mu_p$ are functionally dependent on $B^q$
and together they form a local embedding from the $q$-dimensional ball
$B^q$ to $\mathbb{R}^p$.

Recall that, in the Liouville-integrable symplectic case, the action functions are uniquely determined by the system up to an integral affine
transformation, and they equip the (regular part of the) \textit{base space}, i.e., the space of Liouville tori, with an \textit{integral affine structure}, see, e.g., \cite{DazordDelzant-AA1987,
Duistermaat-AA1980,Zung-Integrable2003}. Similarly, in the presymplectic case, the action functions are also uniquely determined by the system up 
to an integral affine transformation. However, since they are 
over-determined (i.e. functionally dependent) coordinates, they do not
equip the base space (i.e., the space of Liouville tori) of the system 
with an integral affine structure in the usual sense, but rather with what
we will call an \textit{integral co-affine structure}. To be more precise,
let us make the following definition.

\begin{defn}
A \textbf{co-affine chart} of order $p$ on a manifold $Q$ is a chart on 
a ball  $B \subset Q$ together with an injective map $\cA: B \to \bbR^p$.  An \textbf{(integral) co-affine structure} of order $p$ on a manifold $Q$
is an atlas $Q = \cup_i B_i$ of affine charts 
$(B_i \subset Q, \cA_i: B_i \to \bbR^p)$ such that
for any two chart $B_i$ and $B_j$ there is an 
(integral) affine transformation $T_{ij} : \bbR^p \to \bbR^p$
such that $\cA_j = T_{ij} \circ \cA_i$ on the intersection $B_i \cap B_j.$
\end{defn}

\begin{cor} Let $N$ be a Liouville torus of an integrable Hamiltonian system of type $(p,q)$ on a presymplectic manifold $(M,\omega)$ of constant rank
$2q$ ($q < p)$. Then the base space (i.e., space of Liouville tori) of the system in a tubular neighborhood $\cU(N)$ of $N$ is naturally equipped 
with an integral  co-affine structure induced by the system.
\end{cor}

We observe that, similarly to Riemannian structures, 
co-affine structures have a lot of local invariants. 
In particular, one can  talk about the local convexity, the 
curvature of a co-affine structure, and so on.

\section{Action-angle variables on contact manifolds}
\label{section:AAContact}

One of the research problems posed by V.I. Arnold in 1995, 
as listed in the book \textit{Arnold's Problems} \cite{Arnold-Problems2004}, 
was to extend the theory of integrable Hamiltonian systems, in particular the theorem 
on action-angle variables, to contact manifolds. 

In fact, before Arnold posed his problem, several authors,
including Libermann \cite{Libermann-Legendrian1991} in 1991
and Banyaga and Molino \cite{BaMo-Contact1992} in 1992, 
already started the study of integrable systems on contact 
manifolds. Later on many other authors also worked on this
and related problems, see, e.g., Lerman \cite{Lerman-ContactToric2002},
Webster \cite{Webster-Contact2003}, Miranda \cite{Miranda-Thesis2003},
Khesin and Tabachnikov \cite{KhesinTaba-ContactIntegrable2010},
Boyer \cite{Boyer-ContactIntegrable2011},
Jovanovic and Jovanovic \cite{Jovanovic_AAContact2011,JoJo-Contact2015}, etc.

In this section we will explain how to extend the classical
theorem on action-angle variables to contact manifolds. But let 
us first recall the notion of Hamiltonian systems on contact manifolds,
and introduce a natural notion of contact integrable 
systems.

\subsection{Hamiltonian systems on contact manifolds}

Recall that a \textbf{\textit{contact structure}} on a manifold $M$ of dimension $2n+1$ is a regular corank-1 tangent distribution $\xi$ on $M$ such that locally 
(in the neighborhood of every point) 
there is a differential 1-form $\alpha$ such that $\xi = \ker \alpha$ is
the kernel distribution of $\alpha$, and which satisfies the following  nondegeneracy condition:
\begin{equation}
\alpha \wedge (d\alpha)^n \neq 0
\end{equation}
everywhere, which means that $d\alpha$ is nondegenerate on the
distribution $\xi = \ker \alpha$. Such an 1-form
$\alpha$ is called a \textbf{\textit{contact 1-form}}. 

If one multiplies a contact 1-form by an arbitrary non-vanishing function then one gets 
another contact 1-form for the same contact structure. Still, it may happen that
a contact structure does not admit a global contact 1-form, due to orientation problems.

Suppose now that we have a global contact 1-form $\alpha$.
The unique vector field $Z$ such that 
$Z \lrcorner d\alpha = 0$ 
and $\langle \alpha, Z \rangle = 1$ is called the \textbf{\textit{Reeb vector field}}
of the contact form $\alpha$. This vector field is transverse to the contact
distribution $\xi = \ker \alpha$ and is structure-preserving, i.e.
$\cL_Z \alpha = 0.$

Given a function $\phi$ on $M$ which is $Z$-invariant, i.e.
$$ Z(\phi) = 0,$$
there is a unique vector field $X_\phi$
determined by the following conditions:
\begin{equation}
\langle \alpha, X_\phi \rangle = \phi
\end{equation}
(or equivalently, $\langle \alpha, X_\phi -\phi Z \rangle = 0$, i.e.
$X_\phi -\phi Z$ lies in $\xi$), and
\begin{equation}
X_\phi  \lrcorner d\alpha = (X_\phi - \phi Z) \lrcorner d\alpha = - d\phi. 
\end{equation}
The above two equations determine $X_\phi - \phi Z$, and hence $X_\phi$,
uniquely by $\phi$, because of the nondegeneracy of $\alpha$ on $\xi$.
This vector field $X_\phi$ is called the \textbf{\textit{Hamiltonian vector field}} 
of $\phi$ with respect to the contact form $\alpha$. One can check immediately that
$X_\phi$ preserves both $\phi$ and $\alpha$, similarly to Hamiltonian vector fields
on symplectic manifolds:
\begin{equation}
X_\phi(\phi) = 0 \; ; \; \cL_{X_\phi}\alpha = 0.
\end{equation}
Indeed, we have $X_\phi (\phi) = d\alpha(X_\phi, X_\phi) = 0$
and 
$\cL_{X_\phi}\alpha = {X_\phi} \lrcorner d\alpha + d \langle X_\phi, \alpha\rangle
= (-d\phi) + d\phi = 0.$

\subsection{Integrable systems on contact manifolds}

\begin{definition}
An integrable system $(X_1,\hdots,X_p, F_1,\hdots, F_q)$ on a 
contact manifold $(M,\xi)$ is called a \textbf{contact integrable system} on 
$(M,\xi)$ if the vector fields $X_1,\hdots,X_p$ preserve the contact distribution $\xi.$
\end{definition}

\begin{proposition} If $(X_1,\hdots,X_p, F_1,\hdots, F_q)$ is a contact
integrable system on a  contact manifold $(M,\xi)$ then
$$p \leq n+1 = (\dim M +1)/2$$ 
and in a neighborhood $\cU(N)$ of any Liouville torus $N$ on $M$ 
there exists  a $\bbT^p$-invariant contact 1-form 
$\alpha$ on  $\cU(N) \cong \bbT^p \times B^q$ with a Liouville
coordinate system $(\theta_i \pmod 1, z_j)$
such that 
$\xi = \ker \alpha$ and
\begin{equation}
\alpha = \sum_{1}^p a_i(z) d\theta_i + \sum_1^q b_i(z) dz_i.
\end{equation}
\end{proposition}

\begin{proof}
The fundamental conservation property (Theorem \ref{thm:TorusPreservesStructure4}) means that
the Liouville $\bbT^p$-action preserves $\xi$. Define $\alpha$ 
(such that $\xi = \ker \alpha$)
on a section to the torus fibration  and then extend it on $\cU(N)$ via
the action of $\bbT^p$, we get a  $\bbT^p$-invariant contact
1-form $\alpha$. We must have $p \leq n+1$, 
otherwise $\alpha \wedge (d\alpha)^n = 0$ because it does not contain components 
$(d\theta_1 \wedge \hdots \wedge d\theta_{p}) \wedge \hdots$
\end{proof}

\subsection{The generic case}

Generically, the contact distribution $\xi$ is transversal to the
tangent spaces of the Liouville torus $N$ (at some point, 
hence at every point of $N$, because of the $\mathbb{T}^p$-invariance):
\begin{equation}
\xi \pitchfork \left( Span(X_1,\hdots,X_p) = Span (\dfrac{\partial}{\partial \theta_1}, \hdots, \dfrac{\partial}{\partial \theta_p}) \right). 
\end{equation}

Without loss of generality, we may assume that $\xi \pitchfork 
\dfrac{\partial}{\partial \theta_p}.$ 
Dividing $\alpha$ by $\alpha(\dfrac{\partial}{\partial \theta_p})$, 
we may also assume that $\alpha(\dfrac{\partial}{\partial \theta_p}) =1$, 
and so we get 
\begin{equation} \label{eqn:alpha1}
\alpha = d\theta_p + \sum_{1}^{p-1} a_i(z) d\theta_i + \sum_1^q b_i(z) dz_i.
\end{equation}

Notice that in this case $\dfrac{\partial}{\partial \theta_p}  = Z_\alpha$ 
is the \textit{Reeb vector field} of $\alpha$:
$$\dfrac{\partial}{\partial \theta_p} \lrcorner \alpha = 1,\ \ 
\dfrac{\partial}{\partial \theta_p} \lrcorner d\alpha = 
\cL_{\frac{\partial}{\partial \theta_p}}
\alpha - d (\dfrac{\partial}{\partial \theta_p} \lrcorner \alpha) = 0.$$

By forgetting about the coordinate $\theta_p$, i.e. by a natural projection
from $\cU(N)$ to $\bbT^{p-1} \times B^q = \cU(N)/\bbT^1_p$, 
we get a Hamiltonian action of $\bbT^{p-1}$ 
generated by $(\dfrac{\partial}{\partial \theta_1}, \hdots, 
\dfrac{\partial}{\partial \theta_{p-1}})$ 
and momentum map
$(a_1,\hdots,a_{p-1})$ with respect to the symplectic form $d\alpha$.  

Recall that the functions $a_i$ in Formula \eqref{eqn:alpha1}
depend only on the variables $z$, and in particular they are constant
on the Liouville torus $N$. Denote by $c_i = a_i(N)$ their values
on $N$ (so the $c_i$ are constants), and let
$$(\theta_1 \pmod 1,\hdots,\theta_{p-1} \pmod 1, 
z_1,\hdots, z_{p-1}, x_1, y_1, \hdots, x_r, y_r)$$ 
be action-angle variables for this Hamiltonian torus action
(after a change of variables $\theta_1, \hdots, \theta_{p-1}, z_1, \hdots, z_{p-1}$ if necessary), where
$z_i = a_i - c_i$  and $r = (q+1-p)/2$. We can write
$$d\alpha = \sum_1^{p-1} dz_i \wedge d\theta_i + \sum_1^r dx_i \wedge d y_i .$$
Hence
$$\alpha = d\theta_p + \sum_1^{p-1} (z_i+c_i) d\theta_i + \sum_1^r x_i dy_i + \gamma,$$
where $\gamma$ is a closed 1-form which is basic 
with respect to the $\bbT^p$-action.

Put $\gamma = df,$ where $f$ is constant on each Liouville torus, 
and put $\theta_0 = \theta_p + f,$
we obtain the following normal form for $\alpha$:
\begin{equation}
\alpha = d\theta_0 + \sum_1^{p-1} (z_i+c_i) d\theta_i + \sum_1^r x_i dy_i .
\end{equation}

Let $Y = \sum_0^{p-1} f_i(x,y,z) \dfrac{\partial}{\partial \theta_i}$ 
be a vector field which is constant on the tori and which preserves  $\xi$. 
Then $\cL_Y \alpha = g.\alpha$ where $g$ is a $\bbT^p$-invariant function.
Moreover, we have $g = (g.\alpha)(\dfrac{\partial}{\partial \theta_0})
= (\cL_Y \alpha)(\dfrac{\partial}{\partial \theta_0}) = 
\cL_Y(\alpha(\dfrac{\partial}{\partial \theta_0})) = Y(1) = 0$, i.e. 
$Y$ preserves $\alpha.$ Denote $\phi = \alpha(Y).$
Then $\alpha(Y - \phi Z_\alpha) = 0$ and
$(Y - \phi Z_\alpha ) \lrcorner d\alpha = - d\phi,$ 
where $Z_\alpha = \dfrac{\partial}{\partial \theta_0}$ 
is the  Reeb vector field of $\alpha$. 

By definition, $Y$ is the \textit{Hamiltonian vector field} 
of $\phi$ with respect to the contact 1-form $\alpha.$ 
In particular, the vector fields $X_1,\hdots,X_p$ from the beginning
are Hamiltonian with respect to $\alpha$. Thus we have proved the following theorem:

\begin{theorem}[Action-angle variables on contact manifolds]
Let $(X_1,\hdots,X_p,F_1,\hdots,F_q)$ be a contact integrable system on a 
contact manifold $(M,\xi)$ with a Liouville torus $N \cong \bbT^p$ transverse 
to $\xi.$ Then there is a contact 1-form
$\alpha,$ $\ker \alpha = \xi,$ defined in a neighborhood of $N$,
which is invariant with respect to 
the Liouville $\bbT^p$-action and such that the Reeb vector field $Z_\alpha$ of $\alpha$ is one of the generators of this
 $\bbT^p$-action. 
 Moreover, all the vector fields $X_1,\hdots,X_p$ are Hamiltonian with respect to $\alpha$, 
and there is a Liouville coordinate system 
\begin{equation}
(\theta_0 \pmod 1, \hdots, \theta_{p-1} \pmod 1, z_1, \hdots, 
z_{p-1}, x_1, \hdots, x_r, y_1, \hdots, y_r) 
\end{equation}
in a neighborhood $\cU(N) \cong \bbT^p \times B^q$ of $N$ 
($r = (q+1-p)/2$, $x_i,y_i, z_j$ are 0 on $N$) in which the vector fields 
$X_i$ are constant on the Liouville tori and
\begin{equation}
\alpha = d\theta_0 + \sum_1^{p-1} (z_i+c_i) d\theta_i + \sum_1^r x_i dy_i,
\end{equation}
where $c_1, \hdots, c_{p-1}$ are constants. 
\end{theorem}

\begin{remark}
The above theorem is an improvement of the results by Banayaga and Molino 
\cite{BaMo-Contact1992} (the case with $p=n+1, q=n$) and by Jovanovic \cite{Jovanovic_AAContact2011}
(who required more conditions and obtained a less precise normal form).
\end{remark}

In particular, in the case $p=n+1$ (maximal possible) and $q=n$, 
the normal form becomes
\begin{equation}
\alpha = d\theta_0 + \sum_1^{n} (z_i+c_i) d\theta_i
\end{equation}
in a coordinate system $(\theta_0 \pmod 1, \hdots, \theta_n \pmod 1, z_1, \hdots, z_n),$
and the vector fields $X_i$ are of the type (with $\phi_i = \alpha(X_i)$)
\begin{equation}
X_i = Y_{\phi_i} = (\phi_i - \sum_{j=1}^n (z_j+c_j)\dfrac{\partial \phi_i}{\partial z_j}) \dfrac{\partial}{\partial \theta_0} +  \sum_{j=1}^n \dfrac{\partial \phi_i}{\partial z_j} \dfrac{\partial}{\partial \theta_j}.
\end{equation}

\subsection{The non-transversal case} 

As far as we know, the non-transversal case of contact integrable Hamiltonian
systems had not been treated by any author in the literature.

In this case, the Liouville torus $N \cong \bbT^p$ is isotropic, 
i.e. tangent to $\xi$. Therefore, $p \leq n = (\dim M - 1)/2,$ and the Reeb vector field 
$Z_\alpha$ of $\alpha$ is not tangent to $N$. The contact form 
$\alpha$ is preserved by $p+1$ vector fields $Z_\alpha, 
\dfrac{\partial}{\partial \theta_1},\hdots, 
\dfrac{\partial}{\partial \theta_p}$ (and also by the $X_i$). 

Proceeding as in the transversal case, but with $Z_\alpha$ instead of the last 
generator $\dfrac{\partial}{\partial \theta_p}$ of the Liouville $\bbT^p$-action, we get the following
normal form  (with $r = n-p = (q-1-p)/2$):
\begin{equation}
\alpha = dz_0 + \sum_1^p z_i d\theta_i + \sum_1^r x_i dy_i.
\end{equation}

In particular, if $p=n$ (maximal possible) then we have
\begin{equation}
\alpha = dz_0 + \sum_1^n z_i d\theta_i.
\end{equation}

The vector fields $X_i$ are still Hamiltonian, with Hamiltonian functions  $\phi_i = \alpha(X_i):$
\begin{equation}
X_i = Y_{\phi_i} = (\phi_i - \sum_{j=1}^n z_j\dfrac{\partial \phi_i}{\partial z_j}) 
\dfrac{\partial}{\partial z_0} +  \sum_{j=1}^n \dfrac{\partial \phi_i}{\partial z_j} \dfrac{\partial}{\partial  \theta_j}.
\end{equation}

On the other hand, $X_i$ does not contain $\dfrac{\partial}{\partial z_0}$,
i.e. we must have 
$$\phi_i - \sum_{j=1}^n z_j\dfrac{\partial \phi_i}{\partial z_j} = 0,$$
which means that each function $\phi_i$ is homogeneous of degree 1 in the 
variables $z_1,\hdots, z_n$. So these functions $\phi_i$ are linear 
(if they are smooth), 
the quantities $c_{ij} = \dfrac{\partial \phi_i}{\partial z_j}$ are constants,
and the vector fields $X_i = \sum c_{ij} \dfrac{\partial}{\partial  \theta_j}$ are constant vector fields in our normalized coordinate
system (which is quite surprising). 

\section{Action-angle variables on Dirac and Poisson manifolds}
\label{section:AADirac}

\subsection{Dirac manifolds and their submanifolds}

In this subsection, let us briefly recall some basic notions about Dirac manifolds 
and Hamiltonian systems on them 
(see, e.g., \cite{Bursztyn-Dirac2010,Courant-Dirac1990, Kosmann-Dirac2011},
Appendix A8 of \cite{DufourZung-PoissonBook}, and references therein). 
We will also write down some basic results about (co-)Lagrangian submanifolds of Dirac manifolds, which are similar to Weinstein's results
on Lagrangian submanifolds of symplectic manifolds \cite{Weinstein-Symplectic1977},
and which are related to action-angle variables. 

Dirac structures were  first used by Gelfand and Dorfman 
(see, e.g., \cite{GelfandDorfman-Hamiltonian1979,Dorfman-Dirac1987})
in the study of integrable systems, and were formalized by Weinstein and 
Courant in \cite{CourantWeinstein-Dirac1988,Courant-Dirac1990}
in terms of involutive isotropic subbundles of the ``big'' bundle $TM \oplus T^*M$. 
They generalize both (pre)symplectic and Poisson structures, and prove to be a 
convenient setting for dealing with  systems with constraints and reduction problems.

On the direct sum $TM \oplus T^*M$ of the tangent and the cotangent 
bundles of a smooth $n$-dimensional manifold $M$ there is a natural 
fiber-wise indefinite 
symmetric scalar product of signature $(n,n)$ defined by the formula
\begin{equation}
 \langle (X_1,\alpha_1), (X_2,\alpha_2)\rangle = \frac{1}{2}(\langle \alpha_1,X_2\rangle + \langle \alpha_2,X_1\rangle )
\end{equation}
for sections $(X_1,\alpha_1), (X_2,\alpha_2) \in \Gamma (TM \oplus T^*M)$. A vector subbundle 
$\cD \subset TM \oplus T^*M$ is called \textbf{\textit{isotropic}} 
if the restriction of the indefinite scalar product to it is identically zero. 
On the space of smooth sections of $TM \oplus T^*M$ there is an operation, called the
\textbf{\textit{Courant bracket}}, defined by the formula
\begin{equation}
[(X_1,\alpha_1), (X_2,\alpha_2)] := ([X_1,X_2], \cL_{X_1}\alpha_2 - 
{X_2} \lrcorner d\alpha_1),
\end{equation}
where $\cL$ denotes the Lie derivative. 
A subbundle $\cD \subset TM \oplus T^*M$ is said to be \textbf{\textit{closed}} 
under the Courant
bracket if the bracket of any two sections of $\cD$ is again a section of $\cD.$

\begin{defn} A \textbf{\textit{Dirac structure}} on a $n$-dimensional manifold $M$ 
is an isotropic vector subbundle $\cD$ of rank $n$ of
$TM \oplus T^*M$ which is closed under the Courant bracket. If $\cD$ is a Dirac 
structure on $M$
then the couple $(M, \cD)$ is called a {\textbf{\textit{Dirac manifold}}}.
\end{defn}

We will denote the two natural projections $TM \oplus T^*M \to TM$ and 
$TM \oplus T^*M \to T^*M$ by $proj_{TM}$
and $proj_{T^*M}$ respectively. We will also identify $TM$ and $T^*M$
with the subbundles $TM \oplus 0$ and $0 \oplus T^*M$
in $TM \oplus T^*M$.

Grosso modo, a Dirac structure $\cD$ on a manifold $M$ 
is nothing but a singular foliation of $M$ by presymplectic leaves:
the singular \textbf{\textit{characteristic distribution}}  
$\cC = proj_{TM} \cD$ of $\cD$ 
is integrable in the sense of Frobenius-Stefan-Sussmann due to the closedness condition.
On each leaf $S$ of the associated singular \textbf{\textit{characteristic foliation}} 
whose tangent distribution is $\cC$ 
there is an induced differential 2-form $\omega_S$  defined by the formula
\begin{equation}
 \omega_S (X,Y) = \langle \alpha_X, Y \rangle,
\end{equation}
where $X,Y \in \cC_x = T_xS$ and $\alpha_X$ is any element of $T^*_xM$ such that $(X,\alpha_X) \in \cD_x.$ (The pairing $\langle \alpha_X, Y \rangle$ does not depend on the choice of $\alpha_X$
as long as $(X,\alpha_X) \in \cD_x$: if $\alpha'_X$ is another choice then
$\alpha - \alpha' \in \cD_x \cap T_x^*M$, which implies that 
$\langle \alpha_X - \alpha'_X, Y \rangle = 0$ due to the isotropy of $\cD$.
The skew symmetry of $\omega$ is also due to the isotropy of $\cD$:
$ \omega_S (X,Y) +  \omega_S (Y,X) = \langle \alpha_X, Y \rangle +
\langle \alpha_Y, X \rangle = 2 \langle (X,\alpha_X), (Y,\alpha_Y) \rangle = 0$).
Due to the closedness of $\cD,$ the 2-form $\omega_S$ is also closed, i.e.
$(S,\omega_S)$ is a presymplectic manifold. The Dirac structure $\cD$ is uniquely 
determined by its
characteristic foliation and the presymplectic forms on the leaves.

If  $proj_{TM}: \cD \to TM$ is bijective then the characteristic foliation consists 
of just 1 leaf, i.e. $M$ itself,
and $\cD$ is simply (the graph of) a presymplectic structure $\omega$ on $M$: 
$\cD = \{(X, X \lrcorner \omega) \ | \  X \in TM\}.$  On the other hand, if 
$proj_{TM}: \cD \to T^*M$ 
is bijective then $\cD$ is (the graph of) a Poisson structure on $M$, and the 
2-forms $\omega_S$ are nondegenerate,
i.e. symplectic. However, in general, the ranks of the maps 
$proj_{T^M}: \cD \to TM$ and $proj_{T^M}: \cD \to TM^*$ may
be smaller than $n$, and may vary from point to point.

\begin{defn}
A Dirac structure $\cD$ on $M$ is called a \textbf{regular Dirac structure} 
of \textbf{bi-corank $(r,s)$} if there are two 
nonnegative integers $r,s$ such that $\forall x \in M$ we have
\begin{equation}
\dim (\cD_x \cap T_xM) = n - \dim proj_{T^*M}\cD_x = r
\end{equation}
and
\begin{equation}
\dim (\cD_x \cap T^*_xM) = n - \dim proj_{TM}\cD_x  = s.
\end{equation}
\end{defn}

Even if the Dirac structure $\cD$ is non-regular, one can still talk about its 
bi-corank, defined to be the bi-corank of a
generic point in $M$ with respect to $\cD.$ For regular Dirac structures, 
we have the following analog of Darboux's theorem, whose proof is essentially the
same as the proof of the classical local Darboux normal form for symplectic structures.

\begin{prop}[Darboux for regular Dirac] \label{thm:Darboux}
 Let $O$ be an arbitrary point of a $n$-manifold $M$ with a regular 
 Dirac structure $\cD$ of bi-corank $(r,s)$.
Then $n - r - s = 2m$ for some $m \in \bbZ_+ $, and  there is a local coordinate
system $(x_1, \hdots, x_{2m}, y_1,\hdots,y_r,z_1, \hdots, z_s)$  in a neighborhood of $O$, 
such that the local characteristic foliation is of codimension $s$ and given by the local 
leaves 
\begin{equation}
 \{z_1 = const, \hdots, z_s = const\},
\end{equation}
and on each of these local leaves $S$ the presymplectic form $\omega_S$ is given by the 
formula
\begin{equation}
\omega_S = \sum_{i=1}^m dx_{2i-1} \wedge dx_{2i}.
\end{equation}
\end{prop}

In particular, if $\cD$ is regular, then 
the \textbf{\textit{kernel distribution}} given by the kernels of the 
presymplectic forms is regular and integrable, and gives rise
to a foliation called the \textbf{\textit{kernel foliation}} of $\cD$. 
In local canonical coordinates given by Proposition \ref{thm:Darboux},
the kernel distribution is spanned by 
$(\dfrac{\partial}{\partial y_1}, \hdots, \dfrac{\partial}{\partial y_r}).$

\begin{example}
 Given a manifold $L$, a regular foliation $\cF$ on $L$, 
 and a vector bundle $V$ over $L$, put $M = T^*\cF \oplus V$, where
$T^*\cF$ means the cotangent bundle of the foliation $\cF$ over $L$. 
Then $M$ admits the following regular Dirac structure $\cD$,
which will be called the \textbf{\textit{canonical Dirac structure}}: 
each leaf $S$ of the characteristic foliation is of the type
 $S = T^*N \oplus V_N = \pi^{-1} (N)$, where $N$ is a leaf of $\cF$ and 
 $\pi: T^*\cF \oplus V \to L$ is the projection map,
and the presymplectic form on $S = T^*N \oplus V_N$ is the pull-back of the 
standard symplectic form on the cotangent bundle
$T^*N$ via the projection map $T^*N \oplus V_N \to T^*N.$ 
When $\cF$ consists of just one leaf $L$ and $V$ is trivial then
this canonical Dirac structure is the same as the (graph of the) 
standard symplectic structure on $T^*L.$
\end{example}

The notions of Hamiltonian vector fields  and Hamiltonian group actions 
can be naturally extended from the symplectic and
Poisson context to the Dirac context. In particular, we have:

\begin{defn}
 A vector field $X$ on a Dirac manifold $(M,\cD)$ is called a 
 \textbf{Hamiltonian vector field} if 
there is a function $H$, called a \textbf{Hamiltonian function} of $X$, such that 
one of the following two equivalent conditions is satisfied:  \\
i) $(X,dH)$ is a section of $\cD$:
\begin{equation}
 (X, dH) \in \Gamma(\cD).
\end{equation}
ii) $X$ is tangent to the characteristic distribution and 
\begin{equation}
 X \lrcorner \omega_S = - d(H|_S)
\end{equation}
on every presymplectic leaf $(S,\omega_S)$ of it.
\end{defn}

\begin{prop}
 If $X$ is a Hamiltonian vector field of a Hamiltonian function $H$
on a Dirac manifold $(M,\cD),$ then $X$ preserves the Dirac structure $\cD,$ 
the  function $H$, and every leaf of the characteristic foliation.
\end{prop}

\begin{defn} A function $f$ on $(M,\cD)$ is called a \textbf{\textit{Casimir function}} 
if $f$ is a Hamiltonian function of the trivial vector field,
i.e. $(0,df) \in \Gamma(\cD).$ A vector field $X$ on $(M,\cD)$ is called an 
\textbf{\textit{isotropic vector field}} if it is Hamiltonian with respect
to the trivial function, i.e. $(X,d0) \in \Gamma(\cD),$ or equivalently, 
$X$ lies in the kernel of the induced presymplectic forms. 
\end{defn}

Notice that if $X$ is a Hamiltonian vector field of Hamiltonian function $H$, 
$Y$ is an isotropic vector field, and $f$ is a Casimir function, then $X+Y$ 
is also a Hamiltonian vector field of $H$, and $X$ is also a Hamiltonian vector field 
of $H + f$. Modulo isotropic vector fields and Casimir functions, 
the correspondence between Hamiltonian vector fields and
Hamiltonian functions will become bijective.

Remark also that, unlike the Poisson case, not every function on a general 
Dirac manifold can be a Hamiltonian function for some  Hamiltonian vector field. 
A necessary (and essentially sufficient) condition for a function $H$ to be a Hamiltonian
function is that the differential of $H$ must annul the kernels of the induced 
presymplectic forms.

Another interesting feature of general Dirac structures  is that it is easier 
for a dynamical system
to become Hamiltonian with respect to a Dirac structure than with respect to a 
symplectic or  Poisson structure, as the following example shows:

\begin{example} (See \cite{ZungMinh_2D2012}). A local 2-dimensional integrable 
vector field with a hyperbolic singularity
$$X = h(x,y) (\frac{x}{a} \frac{\partial}{\partial x} - \frac{y}{b} 
\frac{\partial}{\partial y}),$$ 
where $a,b$ are two coprime natural numbers  is not Hamiltonian with respect to any 
symplectic or Poisson
structure if $a+b \geq 3,$ but is Hamiltonian with respect to the 
presymplectic structure
$$\omega = x^{a-1} y^{b-1}dx \wedge dy.$$
On the other hand, a local integrable vector field 
$$X = h(y) y \frac{\partial}{\partial x}$$ 
is not Hamiltonian with respect to any presymplectic
structure, but is Hamiltonian with respect to the Poisson structure 
$$y \frac{\partial}{\partial x} \wedge \frac{\partial}{\partial y}.$$
If an integrable vector field on a surface admits both of the above singularities  
then it cannot be Hamiltonian with respect to any presymplectic or Poisson structure, 
but  may  be Hamiltonian with respect to a Dirac structure.
\end{example}

The theory of isotropic, coisotropic, and Lagrangian submanifolds can be naturally extended from the symplectic category to the Dirac
category. However, in the Dirac case, we will have to distinguish between the Lagrangian and the co-Lagrangian
submanifolds (which are the same thing in the symplectic case): A Lagrangian
submanifold will lie entirely on a presymplectic leaf of the Dirac structure and is maximally isotropic there. On the other hand, a co-Lagrangian submanifold will intersect the whole local family of presymplectic leaves, 
with each intersection being isotropic but with tangent spaces having trivial intersection with the kernels of the presymplectic form. Moreprecisely, we have:

\begin{defn} Let $(M,\cD)$ be a Dirac manifold, where $\cD$ is a regular Dirac structure 
of bi-corank $(r,s)$. \\
i)  A submanifold $N$ of  $(M, \cD)$ is called \textbf{\textit{isotropic}} if it 
lies on a characteristic leaf $S$, and the pull-back of
the presymplectic form $\omega_S$ to $N$ is trivial. 
If, moreover, $N$ is of maximal dimension possible, i.e.
\begin{equation}
 \dim N =  \frac{1}{2} \rank \omega_S  + r = \frac{1}{2}(\dim M + r -s),
\end{equation}
then $N$ is called a \textbf{Lagrangian submanifold}.
A foliation (or fibration) on $(M,\cD)$ is called \textbf{\textit{Lagrangian}} 
if its leaves (or fibers) are Lagrangian. \\
ii) A submanifold $L$ of  $(M,\cD)$ is called a \textbf{\textit{co-Lagrangian submanifold}} 
if 
\begin{equation}
\dim L = \frac{1}{2} (\dim M - r + s),
\end{equation}
and for every point $x \in L$  the tangent space $T_xL$ satisfies the following conditions : 
a) $T_xL + proj_ {TM}\cD_x = T_xM;$ b) $T_xL \cap (T_xM \cap \cD_x) = \{0\};$ 
c) $\omega_S|_{T_xL} = 0$ where $S$ is the characteristic leaf containing $x$.
\end{defn}

Observe that  if  $N$ is a Lagrangian submanifold then $T_xN$ contains the kernel of the presymplectic form at
$x$ for every $x \in N,$  this kernel distribution is regular and integrable in $N$, 
and $N$ is foliated by the kernel foliation. If $L$ is a co-Lagrangian submanifold then 
$L$ is also foliated: the foliation on $L$ is the
intersection of the characteristic foliation of $(M,\cD)$ with $L$. Moreover, if 
$N$ is a Lagrangian submanifold and $L$ is a co-Lagrangian submanifold of $(M,\cD)$, 
then
\begin{equation}
 \dim N + \dim L = \dim M.
\end{equation}

The above definition of Lagrangian and co-Lagrangian submanifolds may differ 
from the other definitions in the literature, but they are well-suited for our study of 
action-angle variables. In particular, it is easy to
see that any local Lagrangian foliation in a regular Dirac manifold admits a 
local co-Lagrangian section. We also have the following analogs of some results of Weinstein \cite{Weinstein-Symplectic1977} 
about (co-)Lagrangian submanifolds:

\begin{thm}[Co-Lagrangian embeddings] \label{thm:co-Lagrangian1}
Let $L$ be a co-Lagrangian submanifold of a regular
Dirac manifold $(M,\cD)$. Then there is a foliation $\cF$ on $L$, 
a vector bundle $V$ over $L$, and a Dirac
diffeomorphism from a neighborhood $(\cU(L), \cD)$ of $L$ to an open subset of 
$T^*\cF \oplus V$ equipped with the canonical Dirac structure, which sends 
$L$ to the zero section of $T^*\cF \oplus V.$ 
\end{thm}

\begin{proof}
 Let us first prove the above theorem in the Poisson case: 
 $\cD = \{(\alpha\lrcorner \Pi, \alpha) \ | \ \alpha \in T^*M\}$
is the graph of a regular Poisson structure $\Pi$ on $M$. 
In this case, the bi-corank of $\cD$ is of the type $(0,s)$, and
the leaves of the characteristic foliation are symplectic. 

Denote by $\cF$ the foliation on $L$, which is the intersection of the 
characteristic foliation with $L$:
$T_x\cF = T_xN \cap \cC_xF$ for every $x \in N$, where $\cC$ 
is the characteristic distribution.  Then, via the
symplectic form on $\cC$,  the vector bundle $T^*\cF$ over $L$ is naturally 
isomophic to another vector bundle 
over $L$, whose fiber over $x \in L$ is the quotient space $\cC_x/T_x\cF.$ 
This latter bundle is also naturally 
isomorphic to the normal bundle of  $L$ in $M$. Due to these isomorphisms, 
there is a vector subbundle $E$ over $L$
of $\cC_L = \cup_{x \in L} \cC_x$, such that $\cC_L = T\cF \oplus E,$ and $E$
is Lagrangian, i.e. $E$ is isotropic with respect to the induced 
symplectic forms on $\cC$ and the rank of $E$ is half the rank of $\cC_L$.

At each point $x \in L,$ the set of germs of local Lagrangian submanifolds in 
$M$ which contain $x$ and which are tangent
to $E_x$ at $x$ is a contractible space. 
(By a local symplectomorphism from the characteristic leaf $S$ which contains  $x$
to $T^*\bbR^m$ where $2m = \rank \omega_S$, this space of germs can be identified with 
the space of germs of exact 1-forms  on $(\bbR^m,0)$ whose 1-jets vanish  at the origin). 
Due to this fact, there are no topological obstructions to the existence of a Lagrangian
foliation in a sufficiently  small neighborhood of $L$ which is tangent to $E_x$ at 
every point $x \in L$. Denote by $\cN$ such a Lagrangian foliation. 
Identify $L$ with the zero section of $T^*\cF$. Then, similarly to the proof of  
uniqueness of marked symplectic realizations of  Poisson manifolds  
(see Proposition 1.9.4 of \cite{DufourZung-PoissonBook}), 
one can show that there is a \emph{unique} Poisson isomorphism $\Phi$ from a neighborhood  
of $N$ in $M$ to a neighborhood of $L$  of in $T^*\cF$, which is identity on $L$ and which sends the leaves of $N$ to the local fibers of
$T^*\cF$. $\Phi$ can be constructed as follows: 

Take a local function $F$ in the neighborhood of a point $x \in L$
in $M$, which is invariant on the leaves of the Lagrangian foliation $\cN$. 
Push $F$ to $T^*\cF$  by identifying $L$ with the zero
section of $T^*\cF$ and by making the function invariant (i.e., constant) 
on the fibers of $T^*\cF$. Denote the obtained local
function on $T^*\cF$ by $\tilde{F}$. Now extend the map $\Phi$ from $L$  
(on which $\Phi$ is the identity map) to a 
neighborhood of $L$ by the flows of the Hamiltonian vector fields $X = X_F$ 
and $\tilde{X} = X_{\tilde{F}}$
of $F$ and $\tilde{F}$: if $y = \phi^t_X(z)$ where $z \in L$ and $\phi^t_X$ denotes the time-$t$ flow of $X$, then 
$\Phi(y) = \phi^t_{\tilde{X}} (z).$
One verifies easily that $\Phi$ is well-defined (i.e. it does not depend 
on the choice of the functions $F$), and is a required Poisson isomorphism.

Consider now the general regular Dirac case. Denote by $\cK$ the kernel foliation 
in a small tubular neighborhood $\cU(L)$
of $L$ in $M$ in this case: the tangent space of $\cK$ at each point is the 
kernel of the induced presymplectic form at that point. 
Denote by $\cM \subset \cU(L)$ a submanifold which contains $N$ 
and which is transversal to the kernel foliation.
Then $\cM$ is Poisson submanifold of $(M,\cD)$.
Denote by $\pi_1: \cU(L) \to  \cM$ the projection map (whose preimages are the 
local leaves of the kernel foliation).
 The Dirac structure $\cD$ in $\cU(L)$ is uniquely obtained from the
Poisson structure on $\cM$ by pulling back the symplectic 2-forms from the 
characteristic leaves of $\cM$ to the
characteristic leaves of   $\cU(L)$ via the projection map $\pi_1$
(so that they become presymplectic with the prescribed kernels). 
Denote by $V$ the vector bundle over $L$
which is the restriction of the kernel distribution to $L$.

According to the Poisson case of the theorem, there is a Poisson diffeomorphism 
from $\cM$ to a neighborhood of the zero
section in $T^*\cF$.  Extend $\Phi$ to an arbitrary diffeomorphism 
$\hat{\Phi}$ from $\cU(L)$
to a neighborhood of the zero section in $T^*\cF \oplus V$ which is fiber-preserving 
in the sense  that $\pi_2 \circ \hat{\Phi} = \Phi \circ \pi_1$, where $\pi_2$ denotes the projection $T^*\cF \oplus V \to T^*\cF.$
Then $\hat{\Phi}$ is a  required Dirac diffeomorphism.
\end{proof}

\begin{thm}[Co-Lagrangian sections] \label{thm:co-Lagrangian2}
Let $\cF$ be a regular foliation on a manifold $L$, and $V$ be a vector bundle over $L$. 
Then a section $K$ of the vector bundle $T^*\cF  \oplus V$ equipped with the canonical 
Dirac structure is a co-Lagrangian submanifold of $T^*\cF  \oplus V$  if and only if 
$K = (\theta, v)$, where $\theta \in \Gamma (T^*\cF)$ 
with $d_\cF \theta = 0,$ and $v \in \Gamma(V)$ is arbitrary.
\end{thm}

\begin{proof}
 The proof is the same as in the symplectic case, when $V$ is trivial and $\cF$ consists of 
 just one leaf, i.e. $L$ itself.
\end{proof}

\subsection{Integrable Hamiltonian systems on Dirac manifolds}

The following natural definition is a straightforward generalization of the notion
of integrable Hamiltonian systems from the case of pre-symplectic and Poisson manifolds
to the case of general Dirac manifolds.

\begin{defn}
 An integrable system  $(X_1\hdots,X_p, F_1,\hdots, F_q)$ of type $(p,q)$ 
 on a Dirac manifold $(M,\cD)$ 
 is called an \textbf{\textit{integrable Dirac system}} if the vector fields $X_1,\hdots, X_p$
 preserve the Dirac structure $\cD$. It is called an
 \textbf{integrable Hamiltonian system}  if the vector fields $X_1,\hdots, X_p$ 
 are Hamiltonian, i.e., there are Hamiltonian functions $H_1,\hdots, H_p$ 
 such that $(X_i, dH_i)  \in \Gamma(\cD)$ for $i=1,\hdots,p.$
\end{defn} 

Of course, an integrable Hamiltonian system on a Dirac manifold is 
also an integrable Dirac system, but the converse is not true. 

\begin{prop} \label{prop:LiouvilleToriIsotropic}
Let $N$ be a Liouville torus of an integrable Hamiltonian system
$(X_1\hdots,X_p, F_1,\hdots, F_q)$ with corresponding Hamiltonian
functions $H_1,\hdots, H_p$ on a Dirac manifold $(M,\cD)$. 
Then we have: \\
i) The functions $H_1,\hdots, H_p$ are invariant on the Liouville tori in a tubular neighborhood $\cU(N)$
of $N.$ \\
ii)  The Liouville tori in $\cU(N)$ are isotropic. \\
iii) The functions $H_1,\hdots, H_p$ commute with each other in $\cU(N)$, i.e. their Poisson brackets vanish:
$\{H_i, H_j\} = 0.$
\end{prop}

\begin{proof}
 Recall that, similarly to the case of Poisson manifolds, if $H$ and $F$ are two Hamiltonian functions on a Dirac manifold
with two corresponding Hamiltonian vector fields $X_H$ and $X_F$, 
then their Poisson bracket $\{H,F\} := X_H(F) = - X_F(H) = \omega_S(X_H,X_F)$ (where $\omega_S$ denotes the
induced presymplectic forms) is again a Hamiltonian function whose associated 
Hamiltonian vector field is equal to $[X_H,X_F]$ plus an isotropic vector field.

Since $[X_i,X_j] = 0$ (for any $i,j \leq p$) we have that $X_i(H_j) = \{H_i,H_j\}$ is a Casimir function. In particular,
$X_i(H_j)$ is invariant on the Liouville tori near $N$, because the Liouville tori belong to the characteristic leaves
(because the tangent bundle of the Liouville tori are spanned by the Hamiltonian vector fields $X_1\hdots,X_p$
which are tangent to the characteristic distribution). But the average of $X_i(H_j)$ on each Liouville torus is 0 due to the
quasi-periodic nature of the $X_i$ (Theorem \ref{thm:Liouville})
(see the paragraph after Definition \ref{defn:IntegrableHam2Form} for a more
detailed explanation, the situation here is the same), 
so $X_i(H_j) = 0$ on each Liouville torus, i.e., we have
\begin{equation}
 X_i(H_j)   = \{H_i,H_j\} = 0 \ \text{in} \ \cU(N) \ \forall i=1,\hdots,p, 
\end{equation}
which implies that $H_j$ is invariant on the Liouville tori for all $j =1,\hdots, p.$

The proof of the isotropicness of Liouville tori is absolutely similar to the
presymplectic case: it follows from the equation
$\omega_S(X_i,X_j) = \{H_i,H_j\} = 0$ and the fact that  the vector fields 
$X_1,\hdots,X_p$ span the tangent bundles of the Liouville tori.
\end{proof}

\subsection{Action functions}

We have the following result about Hamiltonianity of the Liouville torus actions
for integrable Hamiltonian systems on Dirac manifolds, similarly to the case
of systems on manifolds with a differential 2-form (Proposition \ref{prop:LiouvilleHamiltonianAction0}):

\begin{thm}[Liouville torus action is Hamiltonian]  \label{thm:LiouvilleHamiltonianAction}
Let $N$ be a Liouville torus of an integrable Hamiltonian system
$(X_1\hdots,X_p, F_1,\hdots, F_q)$  on a Dirac manifold $(M,\cD)$.  
Assume that at least one of the following two additional conditions is satisfied: \\
i) The dimension $\dim (Span_\bbR(X_1(x),\hdots,X_p(x)) \cap (T_xM \cap \cD_x))$ 
is constant in a neighborhood of $N$. \\
ii) The characteristic foliation is regular 
in a neighborhood of $N$. \\
Then the Liouville torus action of the system is a Hamiltonian torus action 
(i.e. its generators are Hamiltonian) in a neighborhood $\cU(N)$ of $N$.
\end{thm}

In particular, if $\cD$ is a Poisson structure then condition i) holds, and if $\cD$ is a presymplectic structure then condition ii) holds, 
and the theorem is valid in both cases. We don't know whether the above theorem is still true in the ``singular'' case when both of the above two conditions
fail or not: we have not been able to produce
a proof nor a counter-example.

\begin{proof}
Let us first prove the theorem under condition ii), i.e. the characteristic foliation if regular. 
Fix a tubular neighborhood $\cU(N) \cong \bbT^p \times B^q$ with a 
coordinate system
$$(\theta_1 \pmod 1, \hdots, \theta_p \pmod 1, z_1,\hdots, z_q)$$ 
in which the vector fields $X_1,\hdots, X_p$ 
are constant on Liouville tori, as given by Theorem \ref{thm:Liouville}.
What we need to show is that $\dfrac{\partial}{\partial \theta_1}$ is a Hamiltonian vector field.
(Then, by similar arguments, the vector fields $\dfrac{\partial}{\partial \theta_2}, \hdots, 
\dfrac{\partial}{\partial \theta_p}$ are also Hamiltonian, so the Liouville torus action is Hamiltonian).
By Theorem \ref{thm:Liouville}, we can write
\begin{equation}
\frac{\partial}{\partial \theta_1}  = \sum_{i=1}^p r_i X_i,
\end{equation}
where the functions $r_i$ are invariant on the Liouville tori. Put
\begin{equation}
 \rho = \sum_{i=1}^p r_i dH_i.
\end{equation}
Then $( \dfrac{\partial}{\partial \theta_1}, \rho ) = 
\sum_{i=1}^p r_i (X_i, dH_i) \in \Gamma(\cD)$ is a section of the Dirac structure. (If $\cD$ is the graph of a presymplectic structure $\omega$
then  $\rho$ is simply the contraction of $\omega$ with 
$\dfrac{\partial}{\partial \theta_1}$).
Since, by Theorem \ref{thm:TorusPreservesStructure4}, 
$\dfrac{\partial}{\partial \theta_1}$ preserves the Dirac structure, 
it also preserves the presymplectic
structure $\omega_S$ of each characteristic leaf $S$, and therefore
$d\rho|_S = d({\dfrac{\partial}{\partial \theta_1}}\lrcorner \omega_S) = 
\cL_{\frac{\partial}{\partial \theta_1}}\omega_S = 0$, i.e., 
the restriction of $\rho$ to each
characteristic leaf  is closed. 

Notice that, by condition ii) and Proposition \ref{prop:LiouvilleToriIsotropic}, each
characteristic leaf in $\cU(N)$ is a trivial fibration by Liouville tori over a disk. The 1-form $\rho$ is
not only closed, but actually exact, on each characteristic leaf, because its pull-back to each Liouville
torus is trivial by construction and by Proposition \ref{prop:LiouvilleToriIsotropic}.  

We can define a Hamiltonian function $\mu_1$ associated to $\dfrac{\partial}{\partial \theta_1}$ as follows (in the presymplectic case, $\mu_1$ is just
a function whose differential is $\rho$, but here the constrution is a bit 
more elaborate because if we do it leaf-wise we have to assure that the end result is a smooth function):

Fix a point $x_0 \in N$, and let $D$ be a small disk containing $x_0$ which is 
transversal to the characteristic foliation. 
Let  $H_1,\hdots, H_p$ be arbitrary Hamiltonian functions associated to $X_1,\hdots,X_p.$ 
For each $y \in \cU(N)$, denote by $y_0$ the  intersection point of the characteristic 
leaf through $y$ in $\cU(N)$ with $D$, and define
\begin{equation} \label{eqn:ActionFunctionIntegral}
\mu_1(y) = \int_{y_0}^y \rho,
\end{equation}
where the above integral means the integral of $\rho$ over a path on a characteristic leaf from $y_0$ to $y$.
The function  $\mu_1(y)$ is well defined, i.e. single-valued and does not depend on the choice of the path,
because of the exactness of $\rho$ on the characteristic leaves. 
It is also obvious that $d\mu_1 = \rho,$ i.e. $\mu_1$ is a Hamiltonian
function of $\dfrac{\partial}{\partial \theta_1}.$

Let us now assume that condition ii) fails, but condition i) holds, i.e.
$$d = \dim (Span_\bbR(X_1(x),\hdots,X_p(x)) \cap (T_xM \cap \cD_x))$$ 
is a constant on $\cU(N).$
Without loss of generality, we can assume that $X_1(x),\hdots,X_{p-d}(x)$ are linearly independent
modulo $Span_\bbR(X_1(x),\hdots,X_p(x)) \cap (T_xM \cap \cD_x))$ for any $x \in \cU(N)$. It implies
that $dH_1 \wedge \hdots \wedge dH_{p-d}(x) \neq 0$ everywhere in $\cU(N).$  By the inverse function
theorem, there exists a disk $D$ which intersects the characteristic leaf $S \ni x_0$ transversally at $x_0,$
and such that the functions $H_1,\hdots,H_{p-d}$ are invariant on $D$. 

Define the action function $\mu_1$ by the same Formula \eqref{eqn:ActionFunctionIntegral} as above, with
$y_0 \in D$. Since
the characteristic foliation in $\cU(N)$ is singular, a general characteristic leaf in $\cU(N)$ can intersect $D$ at 
a submanifold instead of just a point. In order to show that $\mu_1$ is well-defined, we have to check that
if $\gamma$ is an arbitrary oriented curve lying on the intersection of a characteristic leaf $S$ with the disk $D$,
then we have $\int_\gamma \rho = 0.$ But it is the case, because the pull-back of $dH_i$ to $\gamma$ is trivial
for all $i =1,\hdots,p$ by construction. Thus $\mu_1$ is a well-defined single-valued Hamiltonian function
of $\dfrac{\partial}{\partial \theta_1}$, and the theorem is proved.
\end{proof}

The Hamiltonian functions $\mu_1,\hdots,\mu_p$ of the generators 
$\dfrac{\partial}{\partial \theta_1}, \hdots, 
\dfrac{\partial}{\partial \theta_p}$ of the Liouville torus action given 
in Theorem \ref{thm:LiouvilleHamiltonianAction}
will be called \textbf{\textit{action functions}} or \textbf{\textit{action variables}} of the integrable system.
Notice that the action functions are  determined by the system only up to Casimir functions and up to a choice
of the generators of the Liouville torus action (or in other words, a choice of the basis of the torus $\bbT^p$).

\begin{remark}
Another way to obtain action variables in the symplectic case is by 
the following classical integral formula for action functions, which was known already 
to Einstein and other physicists (see, e.g., \cite{BergiaNavarro-EinsteinQuantization2000}),
and which was already used by Mineur in his proof of the existence 
of action-angle variables \cite{Mineur-AA1935-37}:
\begin{equation} \label{eqn:MineurIntegral}
 \mu_1 = \int_{\gamma_1} \alpha,
\end{equation}
where $\alpha$ is a 1-form such that $d\alpha|_S = \omega_S$, and $\gamma_1$ is the loop generated by the
vector field $\dfrac{\partial}{\partial \theta_1}$ on the Liouville torus (for each torus). 
But it is not easy to use  Formula \eqref{eqn:MineurIntegral}
on Dirac manifolds, because of the problem of existence and regularity of $\alpha$ in the Dirac case.  That's why
in the proof of Theorem \ref{thm:LiouvilleHamiltonianAction} we used Formula \eqref{eqn:ActionFunctionIntegral}
instead of Formula \eqref{eqn:MineurIntegral} for the action functions.
In the symplectic case, one can see easily that Formula \eqref{eqn:ActionFunctionIntegral} and
Formula \eqref{eqn:MineurIntegral} give rise to the same action function.
Indeed, by moving the path from $y_0$ to $y$ by the flow of $\dfrac{\partial}{\partial \theta_1}$, we get an annulus with two boundary components, one
is a loop on the Liouville torus $N_0$ wich contains $y_0$ and the other one
is a loop on the Liouville torus $N$ wich contains $y$. By Stokes theorem,
the difference between the values of the action function defined by 
Formula \eqref{eqn:MineurIntegral} on $N$ and on $N_0$ is equal to the
integral of the symplectic form $\omega = d\alpha$ over that annulus. On the
other hand, the contraction of $\omega$ with $\dfrac{\partial}{\partial \theta_1}$ is $\rho$, and so  this integral of $\omega$ over the annulus is equal to the integral in Formula \eqref{eqn:ActionFunctionIntegral}.
\end{remark}

\subsection{Action-angle variables on Poisson manifolds}

Recall that 2-vector field $\Pi$ on a manifold $M$
is called a \textbf{\textit{Poisson structure}} if 
the associated Poisson bracket $\{f,g\} = \langle df \wedge dg, \Pi\rangle$
satisfies the Jacobi identity, or equivalently,  $[\Pi,\Pi] = 0$, 
where $[.,.]$ denotes the Schouten bracket, see, e.g., 
\cite{DufourZung-PoissonBook}. Given a function $G$ on a Poisson
manifold $(M,\Pi)$, the vector field $X = d G \lrcorner \Pi$
is called the \textbf{\textit{Hamiltonian vector field}} 
of $G$ with respect to $\Pi$, and it preserves $\Pi$: $\cL_X \Pi = 0$. 

An integrable system $(X_1,\hdots,X_p,F_1,\hdots,F_q)$ on a Poisson manifold
$(M,\Pi)$ is called \textbf{\textit{Hamiltonian}} with respect to $\Pi$ 
if all the vector fields $X_i = dG_i \lrcorner \Pi$ ($i = 1,\hdots, p$)
are Hamiltonian.  In this case, the functions $G_i$ are automatically
invariant with respect to the Liouville 
$\bbT^p$-action in the neighborhood of each Liouville torus.
Let us recall the following theorem, which is a special case
of Theorem \ref{thm:LiouvilleHamiltonianAction}:

\begin{theorem}[The Liouville torus action is Hamiltonian]
\label{thm:Poisson_HamiltonTorus}
Let $N$ be a Liouville torus of an integrable Hamiltonian system
$(X_1,\hdots,X_p,F_1,\hdots,F_q)$ on a Poisson manifold $(M,\Pi)$.
Then the Liouville $\bbT^p$-action in a neighborhood $\cU(N)$ of $N$ is 
a Hamiltonian torus action.
\end{theorem}

Using Theorem \ref{thm:Poisson_HamiltonTorus} and the same arguments as
in the proof of Weinstein's splitting theorem for Poisson structures
(see, e.g., \cite{DufourZung-PoissonBook,Weinstein-Poisson1983}),
one gets the following result about action-angle variables near a Liouville
torus of an integrable Hamiltonian system on a Poisson manifold,
which was obtained earlier by C. Laurent-Gengoux, E. Miranda and P. Vanhaecke 
\cite{LMV-AA2011} with a much longer proof. 

\begin{theorem}[Laurent--Miranda--Vanhaecke]
\label{thm:AAPoisson}
Let $N$ be a Liouville torus of  an integrable Hamiltonian system
$(X_1,\hdots,X_p,F_1,\hdots,F_q)$ 
of type $(p,q)$ on a twisted Poisson manifold $(M,\Pi)$. Then $q \geq p$, 
and on a neighborhood $\cU(N) \cong \mathbb{T}^p \times D^q$ 
of $N$ there exists a coordinate system
\begin{equation}
(\theta_1 \pmod 1, \hdots, \theta_p \pmod 1,
\ z_1, \hdots, z_{q})
\end{equation}
such that
\begin{equation} \label{eqn:AAPoisson}
\Pi = \sum_{i=1}^p \dfrac{\partial}{\partial z_i} \wedge 
\dfrac{\partial}{\partial \theta_i}  +
\sum_{p < i <j \leq q} b_{ij}(z) \dfrac{\partial}{\partial z_i} 
\wedge \dfrac{\partial}{\partial z_j}
\end{equation}
in $\cU(N)$, the functions $F_1, \hdots, F_q$ do not depend on the variables
$\theta_1,\hdots, \theta_p$, and the vector fields $X_i$ can be written as
$X_i = \sum_j c_{ij}(z) \dfrac{\partial}{\partial \theta_j}$. 
\end{theorem}

\begin{proof}
Denote by $Z_1,\hdots,Z_p$ the generators of the Liouville 
$\bbT^p$-action, which is a Hamiltonian torus action 
by Theorem \ref{thm:Poisson_HamiltonTorus}, and
denote by $(z_1,\hdots, z_p): \cU(N) \to \mathbb{R}^p$ a corresponding
momentum map. 

By induction, we can construct one by one the periodic
functions $\theta_1, \theta_2, \hdots, \theta_p \pmod 1$ on $\cU(N)$
with the following properties:

i) $Z_i (\theta_j) = \delta_{ij}$ is the Kronecker symbol 
($\delta_{ij} = 1$ if $i=j$ and  $\delta_{ij} = 0$ if $i \neq j$).

ii) $\{ \theta_i , \theta_j\}_\Pi = 0$ for all $i,j=1,\hdots,p$.

(It is equivalent to the construction of a coisotropic section $S$ 
to the  Liouville torus fibration in $\cU(N)$, and then put $\theta_i = 0$
on $S$ and then extend them to $\cU(N)$ in a natural way so that
$Z_i(\theta_j) = \delta_{ij}$). 

The functions $z_1,\hdots, z_p$ are 
functionally dependent and $\mathbb{T}^p$-invariant, so we can choose $q-p$ additional 
$\mathbb{T}^p$-invariant coordinates $z_{p+1},\hdots , z_q$ to get a complete coordinate
system on $\cU(N)$. In this coordinate system 
$(\theta_1 \pmod 1, \hdots, \theta_p \pmod 1, z_1, \hdots, z_{q})$, 
all the coefficients of $\Pi$ are $\mathbb{T}^p$-invariant (i.e., they
do not depend on the coordinates $\theta_i$); there is no term of the type
$\dfrac{\partial}{\partial z_i} \wedge 
\dfrac{\partial}{\partial z_j}$ with $i \leq p$ or of the type 
$\dfrac{\partial}{\partial z_i} \wedge 
\dfrac{\partial}{\partial \theta_j}$ 
with $i \neq j$ in the expression of $\Pi$
because  $dz_i \lrcorner \Pi = \dfrac{\partial}{\partial \theta_i}$. 
There is no term
of the type $\dfrac{\partial}{\partial \theta_i} \wedge 
\dfrac{\partial}{\partial \theta_j}$ 
in the expression of $\Pi$ either, because  $\{\theta_i,\theta_j\} = 0$.
That's why $\Pi$ has the expression \eqref{eqn:AAPoisson}. 
\end{proof} 

In particular, in the case of Liouville-integrable Hamiltonian
systems on Poisson manifolds (i.e. the Poisson structure is regular and the dimension
of the Liouville torus is half the dimension of the symplectic leaves) then the
term $\sum_{p < i <j \leq q} b_{ij}(z) \dfrac{\partial}{\partial z_i} 
\wedge \dfrac{\partial}{\partial z_j}$
disappears, and the Poisson structure $\Pi$ has the simple form 
\begin{equation} \label{eqn:AAPoisson2}
\Pi = \sum_{i=1}^p \dfrac{\partial}{\partial z_i} 
\wedge \dfrac{\partial}{\partial \theta_i}
\end{equation}
similarly to the symplectic case. In fact, this Liouville-integrable Poisson 
case is nothing but a parametrized version of the classical Liouville-integrable 
symplectic case. 

\subsection{Full action-angle variables on Dirac manifolds}

As was shown in Proposition \ref{prop:LiouvilleToriIsotropic}, 
Liouville tori of integrable Hamiltonian
systems are isotropic. As a consequence, their dimension satisfies 
the inequality
\begin{equation}
\dim N \leq \frac{1}{2} \rank \omega_S + r,
\end{equation}
where $r = \dim (\cD_x \cap T_xM)$ is the corank of $\omega$ on 
the characteristic leaf containing a Liouville torus $N$. 
The dimension of $N$ is the number of angle variables, 
and also the  number of action variables that we can have. 

In the optimal case, when the above inequality becomes equality, 
i.e., $N$ is a Lagrangian submanifold, then we will say that we have 
a \textbf{\textit{full set of action-angle variables}}.
The word ``full'' means that the presymplectic form in 
this case can be completely described in terms
of action-angle variables. More precisely, we have: 

\begin{thm}[Full action-angle variables]  \label{thm:fullAA}
Let $N$ be a Liouville torus of an integrable Hamiltonian system
\begin{equation}
(X_1,\hdots,X_{m+r},F_1,\hdots,F_{m+s}) 
\end{equation}
on a regular Dirac manifold $(M,\cD)$ of bi-corank $(r,s)$ 
and dimension $n = 2m + r+s$.  
Then the Liouville tori of the system
in a tubular neighborhood $\cU(N)$ are Lagrangian submanifolds of $(M,\cD)$, 
and there is a coordinate system   
\begin{equation}
(\theta_1 \pmod 1, \hdots, \theta_{m+r} \pmod 1, z_1,\hdots, z_{m+s}) 
\end{equation}
on 
\begin{equation}
\cU(N) \cong \bbT^{m+r} \times B^{m+s},
\end{equation}
and action functions 
\begin{equation}
\mu_1 = z_1,\hdots, \mu_m = z_m, \mu_{m+1},\hdots, \mu_{m+r} 
\end{equation}
on  $\cU(N)$, such that the Liouville torus action is generated
by $(\dfrac{\partial}{\partial \theta_1}, \hdots, \dfrac{\partial}{\partial \theta_{m+r}})$,
the functions  $\mu_{m+1},\hdots, \mu_{m+r}$ 
depend only on the  coordinates $z_1,\hdots,z_{m+s}$, 
the characteristic leaves of $\cD$ in $\cU(N)$ are
\begin{equation} \label{eqn:RegularLeaves}
 S_{c_1,\hdots,c_s} = \{z_{m+1}= c_1,\hdots z_{m+s} = c_s\},
\end{equation}
and the presymplectic form $\omega_S$ on each leaf $S= S_{c_1,\hdots,c_s}$ is
\begin{equation} \label{eqn:AA-PresympmlecticFrom}
 \omega_S = (\sum_{i=1}^{m+r} d\mu_i \wedge d \theta_i)|_S .
\end{equation}
\end{thm}

\begin{proof}
The fact that the Liouville tori are Lagrangian is given by Proposition \ref{prop:LiouvilleToriIsotropic} and the
definition of Lagrangian submanifolds. Since the fibration by Liouville tori 
is Lagrangian, we can choose a co-Lagrangian section $D$ of this fibration, 
and a  coordinate system 
$$(\theta_1 \pmod 1, \hdots, \theta_{m+r} \pmod 1, z_1,\hdots, z_{m+s}) $$ 
on $\cU(N)$
such that the leaves of the regular characteristic foliation 
are given by Formula \eqref{eqn:RegularLeaves} and
the  functions $\theta_1, \hdots, \theta_{m+r}$ vanish on $D$, 
i.e. the co-Lagrangian disk $D$ is given by the equation
\begin{equation}
D = \{\theta_1 = 0, \hdots, \theta_{m+r} = 0\}. 
\end{equation}

The existence of action variables $\mu_1,\hdots,\mu_{m+r}$ 
corresponding to the vector fields 
$\dfrac{\partial}{\partial \theta_1}, \hdots,\dfrac{\partial}{\partial \theta_{m+r}}$ 
is given by Theorem \ref{thm:LiouvilleHamiltonianAction}. Without loss of generality, 
we can assume that
$TN$ is spanned by $\dfrac{\partial}{\partial \theta_1}, \hdots,
\dfrac{\partial}{\partial \theta_{m}}$
and the kernel $K = \cD \cap TM.$ Then  
$$d\mu_1 \wedge \hdots \wedge d\mu_m|_S \neq 0$$
everywhere in $\cU(N)$, i.e.,  the functions $\mu_1,\hdots, \mu_m$ 
are functionally independent on the symplectic leaves, but 
$$d\mu_1 \wedge \hdots \wedge d\mu_m \wedge d \mu_{m+i}|_S = 0 \; 
\forall \; i=1,\hdots, r.$$
It follows that we can put $z_1 = \mu_1,\hdots, z_m = \mu_m,$ and choose $z_{m+1},\hdots, z_{m+s}$
to be Casimir functions. 

It remains to prove Formula \eqref{eqn:AA-PresympmlecticFrom}. 
By the invariance of everything with respect to the
Liouville torus action, it is enough to prove this formula at a point $x \in D.$ 
Without loss of generality, we can assume that $\{x\} = N \cap D.$

If $X, Y \in T_xS$ are two vector fields tangent to the characteristic foliation at $x$ 
such that $X,Y \in T_xN$, then
$\omega_S(X,Y) = 0$ due to the isotropy of $N$, and $d\mu_i(X) = d\mu_i(Y) = 0$ for all 
$i=1,\hdots,{m+r},$ which implies
that $(\sum_{i=1}^{m+r} d\mu_i \wedge d \theta_i) (X,Y) = 0.$

If $X,Y \in T_xD \cap T_xS$ then $\omega_S(X,Y) = 0$ because $D$ is co-Lagrangian, and    
$(\sum_{i=1}^{m+r} d\mu_i \wedge d \theta_i) (X,Y) = 0$ because 
$d\theta_i(X) = d\theta_i(Y) = 0$ by construction.

If $X = \frac{\partial}{\partial \theta_j} \in T_xN$ and $Y  \in T_xD \cup T_xS$ then by construction we also have
\begin{multline*}
\omega_S(X,Y) = \omega(\frac{\partial}{\partial \theta_j}, Y) = - d\mu_j (Y) 
= d\mu_j \wedge d \theta_j  (X,Y) =
(\sum_{i=1}^{m+r} d\mu_i \wedge d \theta_i) (X,Y).
\end{multline*}

Since any vector pair $(X,Y) \in (T_xD \cap T_xS)^2$ can be decomposed into a linear combination of pairs of the above types, 
Formula \eqref{eqn:AA-PresympmlecticFrom} is proved.
\end{proof}

\begin{remark}
The above theorem is the analog in the Dirac setting of the action-angle variables theorem for Hamiltonian systems on symplectic
or Poisson manifolds which are integrable à la Liouville. 
In the symplectic case,  the fibers of a regular Lagrangian fibration  with compact fibers are automatically tori, but this fact is no longer
true in the Dirac case: due to the degeneracy of the presymplectic forms on characteristic leaves, 
one can have non-torus Lagrangian fibrations with compact fibers on Dirac manifolds. So on a Dirac manifold 
we need not only a Lagrangian fibration, but also an integrable Hamiltonian system, 
in order to get action-angle variables.
\end{remark}

\subsection{Partial action-angle variables on Dirac manifolds} 
For non-commutatively integrable Hamiltonian systems on symplectic
or Poisson manifolds, there are not enough action-angles variables  
to form a complete  coordinate system, but one can complete these variables by some additional coordinates to form
canonical coordinate systems (see \cite{Nekhoroshev-AA1972,MF-Noncommutative1978,LMV-AA2011}
and previous sections of the present paper). The same is also true in the Dirac setting, 
when the Liouville tori are isotropic but not Lagrangian:

\begin{thm}[Partial action-angle variables]  \label{thm:partialAA}
Let $N$ be a Liouville torus in an integrable Hamiltonian system
$(X_1,\hdots,X_{p},F_1,\hdots,F_{q})$ on a regular Dirac manifold 
$(M,\cD)$ of bi-corank $(r,s)$, such that
the distribution $TN \cap \cD$ is regular of rank $d$ ($0 \leq d \leq r$) 
in a small tubular neighborhood $\cU(N)$  of $N$
fibrated by Liouville tori.  Then there is a coordinate system   
\begin{equation}
(\theta_1 (mod\ 1), \hdots, \theta_{p} (mod\ 1), z_1,\hdots, z_{q}) 
\end{equation}
on 
\begin{equation}
\cU(N) \cong \bbT^{p} \times B^{q}
\end{equation}
and action functions 
\begin{equation}
\mu_1 = z_1,\hdots, \mu_{p-d} = z_{p-d}, \mu_{p-d+1},\hdots, \mu_{p} 
\end{equation}
on  $\cU(N)$, such that the functions  $\mu_{p-q+1},\hdots, \mu_{p}$ 
depend only on the  coordinates 
\begin{equation}
z_1,\hdots,z_{p-d},z_{q-s+1},\hdots, z_{q}, 
\end{equation}
the characteristic leaves of $\cD$ in $\cU(N)$ are
\begin{equation} \label{eqn:RegularLeavesII}
 S_{c_1,\hdots,c_s} = \{z_{q-s+1}= c_1,\hdots, z_{q} = c_s\},
\end{equation}
and the presymplectic form $\omega_S$ on each leaf $S= S_{c_1,\hdots,c_s}$ is of the form
\begin{equation} \label{eqn:AA-PresympmlecticFromII}
 \omega_S = (\sum_{i=1}^{p} d\mu_i \wedge d \theta_i)|_S + \sum_{p-d<i<j\leq q-s-r+d} f_{ij} dz_i \wedge dz_j|_S .
\end{equation}
\end{thm}

\begin{proof}
 The proof is similar to the proof of Theorem \ref{thm:fullAA}. 
 We can assume that $TN$ 
is spanned by $\frac{\partial}{\partial \theta_1}, \hdots,\frac{\partial}{\partial \theta_{p-d}}$
and  $TN \cap \cD.$ Then the action functions $\mu_1,\hdots, \mu_{p-d}$ are independent on the
characteristic leaves, while the remaining action functions $\mu_{p-d+1},\hdots, \mu_p$ are
functionally dependent of $\mu_1,\hdots, \mu_{p-d}$ on each characteristic leaf, i.e., we can write
$\mu_{p-d+1},\hdots, \mu_p$ as functions of $\mu_1,\hdots, \mu_{p-d}, z_{q-s+1}, \hdots, z_{q},$
where the variables $z_{q-s+1}, \hdots, z_{q}$ are Casimir functions, 
and we can put $z_1 = \mu_1,
\hdots, z_{p-q} = \mu_{p-d}$.

Take a section $D$ (of dimension $q$) of the fibration by Liouville tori in $\cU(N)$, with the 
following property: the image of $D$ by the projection $proj: \cU(N) \to \cU(N)/\cK,$
where $\cU(N)/\cK$ denotes the Poisson manifold which is the quotient of $\cU(N)$ by the regular
kernel foliation, is coisotropic of codimension $p-d$, and moreover the intersection of the kernel
foliation with $\cD$ is a regular foliation of dimension $r-d$ on $D$. 
We can choose  the angle variables $\theta_1,\hdots, \theta_p$
so that they vanish on $D$.

The closed 1-forms  $d\theta_i$ do not annul the kernel 
$TM \cap \cD$ in general. 
But they do annul $T_xD \cap \cD_x$ for any point $x \in D$ by construction. 
So for each
$x \in D,$ there are $p-d$ linear combinations $\sum_{j=1}^p c_{ij} d\theta_j$ ($i=1,\hdots, p-d$)
which are linearly independent and which annul the kernel 
$T_xM \cap \cD_x.$ Hence there exist
vectors $$Y_1(x),\hdots, Y_{p-d}(x) \in T_xM$$ 
such that 
$$(Y_i(x), \sum_{j=1}^p c_{ij} d\theta_j(x)) \in \cD_x\; 
\forall\; i=1,\hdots,p-d,$$ 
and these vectors 
$Y_1(x), \hdots,Y_{p-d}(x)$ are linearly independent modulo the kernel $K_x = \cD_x \cap T_xM$
(i.e. no non-trivial linear combination of these vectors lies in $K_x$). 

By the coisotropy property of
$D$ (or more precisely, of the projection of $D$ in $\cU(N)/\cK$), we can choose $Y_1(x),\hdots,
Y_{p-d}(x)$ so that they belong to $T_xD.$ One verifies directly that the distribution $\cY$ on $D$
given by 
$$\cY_x = Span_\bbR(Y_1(x),\hdots,Y_{p-d}(x)) \oplus (T_xD \cap \cD_x)$$ 
is an integrable regular distribution of dimension $p+ r - 2d.$ 

Choose the $q-p+2d-r$ coordinates $z_{p-d+1}, \hdots, z_{q-s-r+d}$ 
in such a way that they are constant on the 
Liouville tori, and also invariant with respect to the
distribution $\cY$. Choose $r-d$ additional coordinates 
$z_{q-s-r+d+1},\hdots, z_{q-s}$ such that
they are also constant on the Liouville tori, 
and such that their differentials when restricted to
the $(r-d)$-dimensional space $T_xD \cap \cD_x$ form a basis of 
the dual space of that space for any point $x \in D$. 

Finally, one verifies that 
$ \omega_S - (\sum_{i=1}^{p} d\mu_i \wedge d \theta_i)|_S$
can be expressed as 
$$\sum_{p-d<i<j\leq q-s} f_{ij} dz_i \wedge dz_j|_S,$$ 
in a way similar 
to the end of the proof of Theorem \ref{thm:fullAA} 
\end{proof}

\section{Torus actions and action-angle variables near singularities}
\label{Section:Singularities}

In mechanics, by \textit{canonical} coordinates many people 
often mean \textit{polar} coordinates which are action-angle variables. 
For example, for the harmonic oscillator with the Hamiltonian
$H = (x^2 + y^2)/2$ in canonical coordinates $(x,y)$, the action-angle variables
$(h,\theta)$ where $h= H$ and $\theta = \arctan(x/y)$ are also often used as 
canonical coordinates. 

In order to normalize a system near a singularity and find polar action-angle 
variables, we can follow the same approach which was used for action-angle variables
near Liouville tori, and which consists of 3 steps:

\begin{itemize}
\item \textit{Intrinsic associated torus actions}. 
Show the existence of a natural torus action (real or complex, formal
or analytical or smooth) which is intrinsically associated to the system
near a singularity. Here the word ``singularity'' may mean either a singular point
of a vector field $X$ (where $X$ vanishes), or a singular orbit of the 
$\bbR^p$-action generated by the vector fields $X_1,\hdots, X_p$ of an integrable 
system $(X_1,\hdots, X_p, F_1,\hdots, F_q)$, 
or a singular level set of the system.

\item \textit{Fundamental conservation property}. Show that the associated torus action is a double commutant in the sense that ``anything'' which is preserved by the system
will also be preserved by the torus action.

\item \textit{Simultaneous normalization/linearization}. Once we know 
that the associated torus action 
preserves an underlying geometric structure (e.g., a symplectic form),
and we can linearize/normalize this geometric structure in an equivariant way with respect
to the torus action (e.g., equivariant version of Darboux's theorem), then we  get
a normal form, which may  contain some polar action-angle 
variables for the system. Even when there are no action variables, the normal form
can still be very interesting.
\end{itemize}

Below let us show how this approach works in various situations.

\subsection{Poincaré--Birkhoff normal forms and torus actions}

Consider a formal or analytic vector field $X$ on $\mathbb{K}^m$,
where $\mathbb{K}$ is $\mathbb{R}$ or $\mathbb{C}$, which vanishes
at the origin $O$ of $\mathbb{K}^m$, 
and consider its Taylor expansion in some (formal or analytic) 
local coordinate system:
\begin{equation}
X = X^{(1)} + X^{(2)} + \hdots 
\end{equation}
The semi-simple part $X^{ss}$ of the linear part $X^{(1)}$
of $X$ is diagonalizable over complex numbers, so we can write
\begin{equation}
X^{ss} = \sum_{i=1}^m \gamma_i z_i \dfrac{\partial}{\partial z_i}
\end{equation}
in some complex coordinate system $(z_1,\hdots, z_m)$. 
A nonlinear monomial vector field 
$z_1^{k_1}\hdots z_m^{k_m}\dfrac{\partial}{\partial z_j}$
is called \textbf{\textit{resonant}} if 
\begin{equation}
[X^{ss}, z_1^{k_1}\hdots z_m^{k_m}\dfrac{\partial}{\partial z_j}] = 0,
\end{equation}
or equivalently, the multi-index tuple $(k_1,\hdots,k_m,i)$
satisfies the \textbf{\textit{resonant relation}}
\begin{equation}
\langle \gamma, k \rangle - \gamma_i  := \sum_{i=1}^m \gamma_i k_i - \gamma_j = 0,
\end{equation}

The classical formal normalization theorems of Poincaré, Birkhoff, Dulac,
Gustavson and other authors say that there exists a coordinate system
in which $X$ commutes with its simisimple linear part:
\begin{equation}
[X, X^{ss}] = 0
\end{equation}
or equivalently, all the non-zero nonlinear terms in the Taylor expansion of
$X$ are resonant. The classical proofs of these theorems are based on the 
method of iterative normalization (elimination of non-resonant terms one by one). When $X$ is Hamiltonian on a symplectic space, then this normalization can be done symplectically, i.e. using canonical systems of coordinates
and canonical transformations. 

In particular, if $X = X_H$ is a Hamiltonian vector field on a 
symplectic space, with the Hamiltonian function  $H = H^{(2)} + h.o.t.$ 
having its quadratic part
\begin{equation}
H^{(2)} =\sum_1^n \gamma_i \dfrac{x_i^2 + y_i^2}{2}
\end{equation}
nonresonant in a canonical coordinate system $(x_i,y_i)$
(which means that the numbers $\gamma_1,\hdots, \gamma_n$ are
incommensurable), then all the resonant terms must be
functions of the action variables $p_i = \dfrac{x_i^2 + y_i^2}{2}$, and
the Birkhoff normal form has the form
\begin{equation}
H = h\left( \dfrac{x_1^2 + y_1^2}{2}, \hdots, \dfrac{x_n^2 + y_n^2}{2} \right),
\end{equation}
i.e., after a formal normalization, in the polar action-angle variables 
$(p_1,\theta_1, \hdots, p_n,\theta_n)$
with $p_i = \dfrac{x_i^2 + y_i^2}{2}$ and $\theta_i = \arctan (x_i/y_i)$,
the non-resonant Hamiltonian function $H$ depends only on the action 
variables $(p_1,\hdots, p_n)$ and is formally integrable.

The minimal number
$d$ such that the semisimple linear part $X^{ss}$ of $X$
can be written in the form
\begin{equation} \label{eqn:d-min}
X^{ss} = \sum_{i=1}^d \lambda_i Z_i,
\end{equation}
where 
\begin{equation}
Z_i = \sum_{j=1}^m a_{ij} z_j \dfrac{\partial}{\partial z_j}
\end{equation}
are diagonal vector fields with integer coefficients $a_{ij} \in \mathbb{Z}$, is called the (complex formal) 
\textbf{\textit{toric degree}} of $X$ at $O$.
The minimality of $d$ is equivalent to the incommensurability of the
numbers $\lambda_1, \hdots, \lambda_d$ in Equation \eqref{eqn:d-min}.

In \cite{Zung-Convergence2005,Zung-Convergence2002} we showed the following results, under the above notations:

i) \textit{$X$ admits a natural \textbf{associated system-preserving
intrinsic torus $\mathbb{T}^d$-action}}, where $d$ is the (complex formal)
toric degree of the system. This torus action is only formal in general, 
even when the system is analytic. In the case of Hamiltonian systems on symplectic manifolds, this torus action is also Hamiltonian.

ii) \textit{The normalization of $X$ à la Poincaré-Birkhoff is equivalent to the linearization of its associated torus action}: 
the system is in its normal form if
and only if this torus action is linear. Instead of step-by-step normalization for $x$, 
one can linearize this torus action using the averaging method
with respect to compact group actions.

iii) $X$ admits a local \textit{analytic} normalization if and only if
its associated torus action is  analytic (and not only formal).

iv) If the system is analytically integrable (i.e., it can be included
in a integrable system $(X_1,\hdots, X_p, F_1,\hdots, F_q)$ 
whose components are all analytic) then its associated torus action is analytic. As a consequence, \textit{any analytic integrable system 
admits a local analytic Poincaré-Birkhoff normalization at 
any singular point}.

This last result  doesn't require any additional
condition on (non)resonance, and is a significant improvement over previous
results obtained by other authors for analytic integrable Hamiltonian systems: R\"ussmann \cite{Russmann-NF1964} (1964, the non-degenerate case with 2 degrees of freedom) Vey \cite{Vey-NF1978} (1978, the non-degenerate case),
Ito \cite{Ito-NF1989} (1989, the non-resonant case),
Ito 1992 \cite{Ito-NF1992} and Kappeler--Kodama--Némethi 1998 \cite{KKN-NF1998}
(the cases with a simple resonance). For integrable Hamiltonian systems there 
are related results by Vey \cite{Vey-Isochore1979} 
and Stolotvitch \cite{Stolovitch-Integrable2000,Stolovitch-Cartan2005}, 
among others, about the existence of local analytic normalization.

Similarly to the case of Liouville tori of integrable systems, the associated 
torus action of a formal or analytic vector field $X$ near a singular point also 
has the fundamental conservation property, even without any integrability condition. In particular, we have:

\begin{theorem}
\label{thm:LocalNF-Conservation1}
Let $\cG$ be an arbitrary formal tensor field which is invariant with respect to
a formal vector field $X$ on $\mathbb{C}^m$ which vanishes at the origin $O$. Then the formal torus $\mathbb{T}^d$-action on $\mathbb{C}^m$ 
associated to $X$ (where $d$ is the 
toric degree of $X$) also preserves $\cG$.
\end{theorem}

\begin{proof} 
The proof is just a simple exercise in linear algebra and representation theory.
We can assume that $X$ is in normal form: $[X,X^{ss}] = 0$, where $X^{ss}$ is
the semisimple linear part of $X$, and is already diagonalized in a coordinate
system $(x_i)$. Then the monomial tensor fields in the coordinates $(x_i)$
are eigenvectors with respect to the Lie derivative operator $\cL_{X^{ss}}$, 
i.e. if $\mathcal{H}$ is a monomial tensor field then $\cL_{X^{ss}}\mathcal{H} = \lambda_\mathcal{H} \mathcal{H}$ for some $\lambda_\mathcal{H}$ 
which is a linear combination
of the eigenvalues of $X$ with integral coefficients. 

Assume that $\cL_{X^{ss}} \mathcal{G} \neq 0$ for some tensor field, we will show that
$\cL_{X} \mathcal{G} \neq 0$. Let $\mathcal{H}$ be the monomial term with a non-zero coefficient $c_\mathcal{H}$ in $\cG$
of lowest lexicographic order with respect to the Jordan decomposition 
$X^{(1)} = X^{ss} + X^{nil}$, such that $\lambda_\mathcal{H} \neq 0$, where $X^{(1)}$ denotes the linear part of $X$
in the coordinate system $(x_i)$ and $X^{nil}$ denotes the nilpotent part of $X^{(1)}$. 
It means that $\mathcal{H}$ has lowest degree in $\cG$, and if $\mathcal{H}'$ is another monomial tensor field such that $(L_n)^k\mathcal{H}' = \mathcal{H}$ for some $k > 0$,
where $L_n$ denotes the nilpotent linear operator 
$\mathcal{A} \mapsto \cL_{X^{nil}} \mathcal{A}$, then $\mathcal{H}'$ does not appear
in the expression of $\cG$, i.e. its coefficient in $\cG$ is zero. 
Due to the nilpotence of $\cL_X - \cL_{X^{ss}}$, the fact that 
 $\cL_X$ commutes with $\cL_{X^{ss}}$, 
and the choice of $\mathcal{H}$, we have that the coefficient of 
$\mathcal{H}$ in $\cL_X \cG$ is exactly equal to the coefficient of 
$\mathcal{H}$ in $\cL_{X^{ss}} \cG$, which is equal to $\lambda_\mathcal{H} \neq 0$.

Thus, in order for $\cL_X \cG$ to vanish, all monomial terms in $\cG$ must be resonant
with respect to $\cL_{X^{ss}}$, i.e., we must have $\cL_{X^{ss}} \mathcal{G} = 0$.

Write $X^{ss} = \sum_{i=1}^d \lambda_i Z_i$, where $d$ is the toric degree of $X$,
$Z_i = \sum a_{ij} x_j \partial x_j$ are diagonal linear vector fields with integral coefficients, and $\lambda_1,\hdots, \lambda_d$ are incommensurable numbers. The vector fields $\sqrt{-1} Z_1,\hdots, \sqrt{-1} Z_d$ are the generators of the associated
torus action of $X$ (see \cite{Zung-Convergence2002,Zung-Convergence2005}), and for each
monomial tensor field $\mathcal{H}$ we have 
$$\cL_{X^{ss}} \mathcal{H} = \sum_i \lambda_i \cL_{Z_i} 
\mathcal{H} = \sum_i c_i m_i \mathcal{H},$$ 
where $m_i \in \mathbb{Z}$ is the eigenvalue of $\mathcal{H}$
with respect to $\cL_{Z_i}$. Since the numbers $c_i$ are incommensurable, we have that
$\sum_i c_i m_i = 0$ if and only if $m_i =0$ for every $i$, i.e. $\cL_{Z_i} \mathcal{H} =0$
for every $i$. Thus, from $\cL_{X^{ss}} \mathcal{G} = 0$ we get $\cL_{Z_i} \mathcal{G} =0$
for every $i$, which means that $\cG$ is invariant with respect to the associated torus action.
\end{proof}

In particular, if $X = X_H$ is a Hamiltonian vector field, and $\cG$ is the symplectic form,
then the associated torus action of $X$ automatically preserves $\cG$, i.e. it is a Hamiltonian torus action. Because of this, Birkhoff normalization and Poincaré-Dulac normalization are essentially the same theory (see \cite{Zung-Torus2006}).

The toric degree $d$ discussed above is also called \textit{complex} toric degree,
because even when $X$ is a real analytic
vector field on $\mathbb{R}^m$, its associated torus
$\mathbb{T}^d$-action acts on the complexified space $\mathbb{C}^m$
and not on the real space $\mathbb{R}^m$ in general. Only a subtorus
of this complex torus action, say 
$\mathbb{T}^{d_r} \subset \mathbb{T}^d$ acts in the real space, and
the (largest possible) number $d_r$ is called the 
\textbf{\textit{real toric degree}} of $X$ at the origin $O$.

For example, if $X = X_H$ is a real integrable Hamiltonian vector field with a nonresonant singularity at $O$, then according to a classification
by Williamson (see., e.g.,  \cite{Eliassion-NF1990,
Miranda-Thesis2003,Zung-Integrable1996}) and the local normalization theorem, the function $H$
can be written in some local analytic system of canonical coordinates 
as a real analytic function 
$$H = h(\mu_1, \hdots, \mu_n)$$
of the components of $\mu_i$, which belong to 3 following types:
\begin{itemize}
\item \emph{Elliptic} $\mu_j = \dfrac{x_j^2 + y_j^2}{2};$
\item \emph{Hyperbolic} $\mu_i = x_iy_i;$
\item \emph{Focus-Focus} $\mu_{i} = x_iy_i + x_{i+1} y_{i+1}, 
\  \mu_{i+1} = x_iy_{i+1} - x_{i+1}y_i.$
\end{itemize}
The \textbf{\textit{Williamson type}} 
of $X_H$ at $O$ is the triple $(k_e,k_h,k_f)$, where
$k_e$ (resp., $k_h$, $k_f$) is the number of elliptic (resp., hyperbolic,
focus-focus) components of $X$ at $0$. The complex
toric degree of $X$ in this case is $n = k_e+k_h+2k_f $, but its real toric degree is only
$k_e + k_f$: there is a local real analytic effective Hamiltonian torus 
$\mathbb{T}^{k_e+k_f}$ which preserves the system near $O$, and that number
is maximal possible. If $\cG$ is a real analytic tensor field (or more generally, a subbundle of a natural bundle) which is preserved by $X$,
then its complexification  is invariant under the torus $\mathbb{T}^d$-action
in the complexified space, and hence $\cG$ is also invariant under the
associated action of $\mathbb{T}^{d_r}$ (where $d_r$ is the real toric degree)
in the real space (which is a subaction of the associated torus action). 

\subsection{Torus actions near singular orbits and fibers}

In \cite{Zung-Toric2003}, it was shown that, under a very mild
so-called \textit{finite-type condition}, near any 
$k$-dimensional compact orbit of a real analytic integrable system $(X_1,\hdots,X_p,F_1,\hdots,F_q)$ there is a 
$\mathbb{T}^k$-action which preserves the system and is transitive on the orbit. Together with the \textit{\textbf{transverse toric degree}} $d$, which is the toric degree of the reduced system 
with respect to this $\mathbb{T}^k$-action
we get an effective torus $\mathbb{T}^{k+d}$-action in a complexified neighborhood of the compact orbit, which plays
the role of the \textit{associated torus action} in this case,
and Theorem \ref{thm:LocalNF-Conservation1} naturally extends
to this associated torus action. As usual, the proof in the analytic case is simpler than in the smooth case, because we can make use of the natural filtration by the degree of the monomial terms. 

For smooth but non-analytic integrable system, it does not make sense to talk about complexification, the torus action acts
only in the real space (and so its dimension is smaller than the
complex toric degree in general), the problem of normalization is harder in general, and one needs some nondegeneracy conditions
(not just the finite-type condition which is too weak). Torus actions and partial action-angle variables for smooth
integrable systems near singularities have been obtained
by many authors, most notably Eliasson \cite{Eliassion-NF1990}, Dufour--Molino \cite{DufourMolino-AA1990} and Banyaga--Molino
\cite{BaMo-Contact1992}
(polar action-angle coordinates for elliptic singularities, on symplectic and contact manifolds), 
Miranda--Zung \cite{MirandaZung-NF2004} (torus action and linearization near a nondegenerate compact orbit), and in our papers \cite{Zung-Integrable1996,Zung-Torus2006}. Using topological arguments, it was shown in \cite{Zung-Integrable1996} (see also
\cite{Zung-Torus2006}) that, under some nondegeneracy conditions (which may be weakened) near a singular fiber of rank $k$ of
a proper integrable Hamiltonian system on a symplectic manifolds there is a system-preserving Hamiltonian $\mathbb{T}^k$ action, which give rise to partial action-angle variables $(\mu_i,\theta_i)$: the symplectic
form $\omega$ has the form 
$$ \omega = \sum_{i=1}^k d\mu_i \wedge d\theta_i + \beta,$$
where $\beta$ does not depend on the variables $(\mu_i,\theta_i).$

Naturally, results of Section \ref{section:Conservation} about
the fundamental conservation property extend to these smooth torus
actions near singular orbits and fibers of proper integrable systems, because these torus actions are parts of Liouville torus actions outside of the singularities.

\section{Some final remarks}
\label{Section:StochasticQuantum}

In this paper, we studied only finite-dimensional classical dynamical systems
generated by vector fields, but the same idea about associated torus actions
playing the role of double commutants should also work in many other contexts. 
Let us mention here a few.

$\bullet$ \textit{Systems on Nambu manifolds}. See \cite{DufourZung-PoissonBook}
for an introduction to Nambu structures, and \cite{ZungMinh--Foliations2013}
for a more precise relationship between Nambu structures and singular foliations.

Let $(X_1,\hdots,X_p,F_1,\hdots,F_q)$ be an integrable system on 
a manifold $M$ and $\cD$ be a smooth distribution on $M$ such that 
the vector fields $X_i$ are tangent to $\cD$ and preserve $\cD$. 
Then it follows easily from the fundamental conservation property that
any Liouville torus $N$ of the system admits a tubular neighborhood 
$\cU(N)$ which can be decomposed as
\begin{equation}
\cU(N) \cong \bbT^p \times B^q, 
\end{equation}
with $\cD = T\bbT^p \oplus \cD_1$,
where $\cD_1$ is a smooth distribution on $B^q.$

In the case when $\cD$ is the associated distribution of an
(almost) Nambu structure $\Lambda$ and $\cL_{X_i} \Lambda = 0,$ 
then from the above splitting result for $\cD$
we obtain the following splitting formula for $\Lambda:$
\begin{equation}
\Lambda = \dfrac{\partial}{\partial \theta_1} 
\wedge \hdots \wedge \dfrac{\partial}{\partial \theta_p} \wedge \Pi,
\end{equation}
where $\Pi$ is an (almost) Nambu structure on $B^q$, and $\theta_1,\hdots,\theta_p$ are action variables. \\

$\bullet$ Torus actions in normalizations of diffeomorphims
and in the singular perturbation method (renormalization).
Torus actions appear naturally not only in dynamical systems 
generated by vector fields, but also in diffeomorphisms, renormalization method, and so on, see, e.g., \cite{Raissy-Torus2010,Chiba-Torus2009}.
It is natural to expect that these natural torus actions also
possess the fundamental conservation property. 

$\bullet$ \textit{Infinite-dimensional integrable systems}.
If we have an integrable Hamiltonian system on an infinite dimensional symplectic manifold modeled on a separable Hilbert space, for example the periodic KdV or the integrable non-linear Schr\"odinger equation, then one observes the following phenomena (see, e.g., \cite{KaPo-KAM2003,KLPZ-NLS2009,KuPe-Vey2010}): 

i) The regular and singular elliptic level sets are compact tori homeomorphic 
(with the subspace topology) to $\mathbb{T}^\mathbb{N}$  with product topology or to a finite-dimensional torus. The tori
$\mathcal{T}_c = \{z = (z_n) \in \ell_2 \mid  z_i \in \mathbb{C}, |z_i|^2 = c_i\}$, where $c = (c_n)$ is a sequence of nonnegative real numbers whose sum converges, which appear in infinite-dimensional harmonic oscillators, 
are typical examples of such compact infinite-dimensional tori. (It is a simple exercise in elementary topology to show that the subspace topology of each such torus is the same as the product topology). 

ii) There is an analytic infinite-dimensional Hamiltonian 
torus action near each Liouville or elliptic torus, which preserves the system
and which gives rise to an infinite-dimensional normal form à la Birkhoff,
with an infinite number of action and angle variables.

$\bullet$ \textit{Quantum and semi-classical integrable systems}. 
In a quantum mechanical system, angle variables are phases, 
action variables correspond to quantum numbers, and associated torus
actions are periodic unitary transformations which preserve the system.
(See, e.g., \cite{LLH-QuantumOscillatorAA1997} for the example of
quantum harmonic oscillator). The fundamental conservation property here 
means that any quantum operator which commutes with the system will also
commute with the phase shift operators.  In the semi-classical case,
one has semi-classical Birkhoff normal forms 
(see, e.g., \cite{San-Semiclassical2006}), and there should exist associated
torus actions, which preserve everything which is preserved by the 
semi-classical system, and whose linearization corresponds to the
semi-classical normalization of the system.

\section*{Acknowledgement}

This paper is based on an unpublished preprint 
``Action-Angle variables on Dirac manifolds'' (arXiv:1204.3865) 
written by the author in 2012, which contains the
central idea of the paper, namely action-angle normal forms 
can be easily obtained from the fundamental conservation property of the 
associated torus actions: \textit{everything which is preserved by 
the system is also preserved by the associated torus actions}.

The results of this paper had been exposed in several talks at seminars and 
conferences, including the conference ``Geometry and Mechanics''
at Institut Henri Poincaré, Paris, 22-24/Nov/2014 in honour of C.-M. Marle,
and the ``Workshop on Poisson Geometry and Mathematical Physics'' at
Chern Institute, Tianjin, 04-08/Jan/2016. I would like to thank
many colleagues for their interest in this work and for
various discussions, and especially to Tudor Ratiu for encouraging me to 
write up this paper. 

I finally found time to work out various details and to 
write up this paper in the first half of 2017, 
thanks to a stay at the School of Mathematical Sciences, Shanghai 
Jiao Tong University as a visiting professor. I would like to thank Shanghai 
Jiao Tong University, 
the colleagues at the School of Mathematics of this university, and 
especially Tudor Ratiu, Jianshu Li, and Jie Hu for the invitation, 
hospitality and excellent working conditions. 

I would like also to thank the referees for many remarks
which helped me to improve the presentation of the paper. 

\section*{Compliance with ethical standards}

This research did not involve any animals or human participants.

As far as I know, there is no conflict of interest; the main results presented in this
paper are original and have not been submitted for publication in any other journal.

\end{document}